\documentclass{elsarticle}
\usepackage{amsmath,amsthm,amssymb}
\usepackage{amssymb}
\usepackage{graphicx}
\usepackage{subfigure}
\usepackage{amsmath}
\usepackage{epstopdf}
\usepackage{color}
\usepackage{multirow}
\usepackage{lineno,hyperref}
\usepackage{stmaryrd}
\allowdisplaybreaks
\numberwithin{equation}{section}
\newtheorem{theorem}{Theorem}[section]
\newtheorem{lemma}{Lemma}[section]

\newtheorem{remark}{Remark}[section]

\biboptions{sort&compress}
\modulolinenumbers[5]

\journal{Journal of \LaTeX\ Templates}

\begin{document}

\begin{frontmatter}

\title{Adaptive FEM  with optimal convergence rate for non-self-adjoint eigenvalue problems \tnoteref{mytitlenote}}
\tnotetext[mytitlenote]{Projects supported by the Foundation of Guizhou Provincial Department of Science and Technology (Qiankehe Basic Research (MS) 〔2026〕 No. 434) and the National Natural Science Foundation of China (Grant Nos. 11561014, 11761022).}

\author[mymainaddress]{Shixi Wang}
\ead{wangshixi@gznu.edu.cn}

\author[mymainaddress]{Hai Bi}
\ead{bihaimath@gznu.edu.cn}

\author[mymainaddress]{Yidu Yang\corref{mycorrespondingauthor}}
\cortext[mycorrespondingauthor]{Corresponding author}
\ead{ydyang@gznu.edu.cn}

\address[mymainaddress]{School of Mathematical Sciences, Guizhou Normal University, Guiyang 550025, China}

\begin{abstract}
In this paper, we first discuss the optimal convergence of the adaptive finite element methods for non-self-adjoint eigenvalue problems. We present new theoretical error estimators and computable error estimators 
for multiple and clustered eigenvalues with the help of the error estimators of finite element solutions
 for the corresponding source problems, and prove the equivalence between these two estimators.
We propose an adaptive algorithm for the eigenvalue cluster and demonstrate that it achieves the optimal convergence rate. 
We also provide numerical experiments to support our theoretical findings. 
\end{abstract}

\end{frontmatter}

\section{Introduction}\label{sec1}
\indent Since the pioneering work of Babu\v{s}ka and Rheinboldt\,\cite{babuska1978} in 1978, the theory and algorithms of adaptive finite element methods (AFEM) have seen significant advancements; see monographs \cite{ainsworth2000,verfurth2013} and recent review papers \cite{Nochetto2012Primer,chamoin2023,bonito2024} for instance.
A posteriori error analysis and adaptive computation of eigenvalue problems have become a hot topic in the academic community (see \cite{duran1999,duran2003,larson2000,oden2003,gedicke2014,carstensen2011,boffi2019,heuveline2001,garau2009,rannacher2010,cances2020,carstensen2024adaptive,boffi2025adaptive}), particularly regarding the adaptive computation of multiple and clustered eigenvalues attracts much attention in the academic community.\\
\indent There have been several contributions to the proof of convergence and optimal rates for AFEM for self-adjoint eigenvalue problems.
In 2008, Dai~et~al.~\cite{dai2008} first proved the optimal convergence rate of AFEM for simple eigenvalues and eigenfunctions.
In 2015, Dai~et~al.~\cite{dai2015} first proposed a theoretical estimator for multiple eigenvalues and demonstrated the convergence and quasi-optimal complexity of AFEM. 
They showed that the theoretical error estimator is equivalent to a standard computable estimator on sufficiently fine mesh.
Subsequently, Gallistl~\cite{gallistl2015} studied clustered eigenvalues for the linear FEM and proved the convergence and quasi-optimal complexity of AFEM. 
This investigation has been further extended and developed.
Bonito and Demlow~\cite{bonito2016} established the convergence with optimal rate of AFEM for clustered eigenvalues using elements of arbitrary polynomial degree, Boffi~et~al.~\cite{boffi2017} developed an $h$-adaptive mixed FEM that achieves optimal convergence rates for discretizing Laplace operator eigenvalue clusters, and Gallistl~\cite{gallistl2014} analyzed a nonconforming AFEM of eigenvalue clusters of
the Laplacian and of the Stokes system.
The equivalence argument between the computable error estimator and the theoretical error estimator proposed in \cite{gallistl2015,bonito2016} is general and concise.
The analysis method in \cite{gallistl2014,boffi2017} for the optimality of AFEM is also general and easy to apply to other eigenvalue problems.
As far as the authors' best knowledge, however, there is currently no literature reporting the optimal convergence rate of AFEM for non-self-adjoint eigenvalue problems. 
The purpose of this paper is to fill the gap.\\
\indent Our research relies on the discussion of the optimal convergence rate of AFEM for the corresponding source problems.
For second-order elliptic PDEs, the important literature, see \cite{dorfler1996,mns,binev2004,stevenson07,cascon2008,feischl2014,babuska1978,ainsworth2000,verfurth2013,axioms,buffa2016,hu2018,feischl2022,veiga2023,chamoin2023,bonito2024}, provide relevant studies for adaptive Galerkin methods. \\
\indent Compared with the self-adjoint eigenvalue problems, the ascent of eigenvalues in non-self-adjoint eigenvalue problems may be greater than one, resulting in defective eigenvalues, and eigenfunctions no longer form a complete orthogonal system.
Moreover, generalized eigenfunctions of order greater than 1 appear (see pages 683 and 693 in \cite{babuska1978}), so the relation $a(u,v)=\lambda (u,v)_0$ (see (\ref{eqn:2.3})) does not hold. This poses challenges for theoretical analysis.\\
\indent Feischl~et~al.~\cite{feischl2014} proved convergence with optimal algebraic convergence rates for AFEM of general linear second-order elliptic operators.
By the argument in Yang~et~al.~\cite{yang2024}, the a posteriori error estimates of the Galerkin solution of the eigenvalue problem are attributed to the error estimates of the Galerkin solution of the corresponding source problem. 
Based on \cite{feischl2014,yang2024}, we present new theoretical error estimators and computable estimators for multiple eigenvalues and eigenvalue clusters, and we prove their equivalence using the arguments in \cite{gallistl2015,bonito2016}.  
We then propose an adaptive algorithm and prove the optimal convergence rate of AFEM for the non-self-adjoint eigenvalue problems by following the lines of the proof of the optimal convergence rate of AFEM in \cite{gallistl2015,boffi2017} and using the arguments as in \cite{cascon2008,feischl2014}.\\
\indent The outline of the rest of the paper is as follows. Section \ref{sec2} is dedicated to the preliminaries necessary for our work. 
In section \ref{sec3}, we establish the equivalence of the computable error estimator and the theoretical error estimator. 
Section \ref{sec4} analyzes the optimality of AFEM. Finally, numerical examples are presented in section \ref{sec5}.\\
\indent Throughout this paper, the letter $C$ (with or without subscripts) denotes a generic positive constant
that may depend on domain $\Omega$ and the initial mesh $\mathcal{T}_{0}$ 
but not on the mesh $\mathcal{T}$ which is a refinement of $\mathcal{T}_{0}$ 
and not on the eigenvalue cluster of interest.
We use the notation $a \lesssim b$ to indicate that $a \le Cb$,
and write $a \simeq b$ when $a \lesssim b$ and $b\lesssim a$ for simplicity.

\section{Preliminaries}\label{sec2}
Let $\Omega \subset \mathbb{R}^d (d = 2,3)$ be a polyhedral bounded domain. 
For $\omega \subset \Omega$, we denote by $H^1(\omega)$ the usual Sobolev space endowed with the norm
$
\|v\|^2_{1, \omega}:=\|v\|_{L^2(\omega)}^2+\|\nabla v\|_{L^2(\omega)}^2
$.
Denote $\|\cdot\|_1=\|\cdot\|_{1, \Omega},\|\cdot\|_{0, \omega}=\|\cdot\|_{L^2(\omega)}$ and $\|\cdot\|_0=\|\cdot\|_{0, \Omega}$.
Let $V=H_{0}^{1}(\Omega)=\left\{v: v \in H^1(\Omega),\left.v\right|_{\partial \Omega}=0 \right\}$.
Let $(u, v)_0  =\int_{\Omega} u \overline{v} \mathrm{~d}x$ denote the $L^2(\Omega)$ inner product.\\
\indent Consider the eigenvalue problem for the following non-self-adjoint elliptic operator
\begin{align}
\label{eqn:2.1}
\mathcal{L} u := -\nabla \cdot(\boldsymbol{A} \nabla u)+\boldsymbol{b} \cdot \nabla u+cu&=\lambda u \quad \text { in } \Omega, \\
\label{eqn:2.2}
 u&=0 \quad \text { on } \partial \Omega,
\end{align}
where $\boldsymbol{A}=(a_{ij}(x))_{d\times d}$ is a symmetric matrix, 
$\boldsymbol{b}=\boldsymbol{b}(x)\in$ $L^{\infty}(\Omega)^d$ is a vector and $c=c(x) \in L^{\infty}(\Omega)$ is a scalar. 
For the sake of exposition, we assume that $\boldsymbol{A}$, $\boldsymbol{b}$ and $c$ are piecewise sufficiently smooth functions.\\
\indent The weak formulation of the problem (\ref{eqn:2.1})--(\ref{eqn:2.2}) reads as:
Find $\lambda \in \mathbb{C}$, $0\not= u \in V$, such that
\begin{eqnarray}\label{eqn:2.3}
a(u,v)=\lambda (u,v)_0\quad \forall v \in V,
\end{eqnarray}
where 
\begin{equation*}
\begin{aligned}
a(u, v) & :=(\boldsymbol{A}\nabla u \cdot \nabla v)_0 + (\boldsymbol{b} \cdot \nabla u + c u,v)_0 \quad \text{ for } u,v\in V. \\
\end{aligned}
\end{equation*}

The sesquilinear form $a(\cdot, \cdot)$ is bounded with
\[
|a(u, v)| \lesssim \|\nabla u\|_{0}\|\nabla v\|_{0}  \quad \forall u, v \in V.
\]

We assume that the coefficients satisfy the conditions of Proposition 31.8 in \cite{ern2021}. Then from Proposition 31.8 in \cite{ern2021}, $a(\cdot, \cdot)$ is V-coercive:
\begin{equation}\label{coercive}
a(u, u) \gtrsim \|\nabla u\|_{0}^2 \quad \forall u \in V.
\end{equation}

We define the quasi-norm $\|v\|_V:=a(v,v)^{1/2}$ and we have
$$\|v\|_V  \simeq \|v\|_1\quad \forall v\in V.$$

\indent The source problem associated with (\ref{eqn:2.3}) is as follows: Given $f\in L^2(\Omega)$, find $w\in V$  satisfying
\begin{eqnarray}\label{eqn:2.4}
a(w,v)=(f,v)_0\quad \forall v \in V.
\end{eqnarray}

\indent Given an initial simplicial mesh $\mathcal{T}_{0}$ which is regular as in \cite{ciarlet1978},
the mesh sequence $\{\mathcal{T}_{\ell}\}_{\ell\in\mathbb N}$ with mesh size $h_\ell$ is generated by using the newest vertex bisection (NVB) refinement strategy.
Let $$\mathbb T = \{\mathcal T_\ell: \mathcal T_\ell \text{ is a refinement of the initial mesh } \mathcal T_{0}\}$$ be the set of meshes.
For $\ell\in\mathbb N$, $\mathcal T_{\ell}\in \mathbb T$ is regular and $\gamma$-shape regular, i.e.,
\[\gamma^{-1}|T|^{1/d}\leq \mathrm{diam}(T) \leq \gamma|T|^{1/d},\]
where $|T|$ and $\mathrm{diam}(T)$ denote the $d$-dimensional volume  and diameter of $T\in\mathcal T_{\ell}$, respectively.\\
\indent For $p \geq 1$, we denote the set of the piecewise polynomials of degree $\leq p$ on $\mathcal{T}_\ell$ by $\mathcal{P}^p(\mathcal{T}_\ell)$.
Let $V_{\ell}\subset V$ be the conforming Lagrange finite element space defined on $\mathcal T_\ell\in \mathbb T$.
 $V_\ell:= \mathcal{P}^p (\mathcal{T}_\ell) \cap V$.\\
\indent The discrete form for (\ref{eqn:2.3}) is given by:
Find $\lambda_{\ell} \in \mathbb{C}$, $0\not= u_{\ell} \in V_{\ell}$, such that
\begin{eqnarray}
\label{eqn:2.5}
a(u_{\ell},v)=\lambda_{\ell} (u_{\ell},v)_0\quad \forall v \in V_{\ell}.
\end{eqnarray}
\indent  The discrete form for (\ref{eqn:2.4}) is given by: Find $w_{\ell}\in V_{\ell}$, such that
\begin{eqnarray}
    \label{eqn:2.6}
    a(w_{\ell},v)=(f,v)_0\quad \forall v \in V_{\ell}.
\end{eqnarray}

\indent Since $a(w,v)$ is coercive, $a(w,v)$ satisfies the Lax-Milgram lemma, and (\ref{eqn:2.4}) and (\ref{eqn:2.6}) are well-posed. 
Then, we can define the solution operators $K: L^2(\Omega)\to V$ and $K_\ell: L^2(\Omega)\to V_{\ell}$ by
\begin{align}
\label{eqn:2.7a}
&a(Kf,v)=(f,v)_0\quad \forall v \in V,\\
\label{eqn:2.7}
&a(K_\ell f,v)=(f,v)_0\quad \forall v \in V_{\ell}.
\end{align}
And there holds
\begin{eqnarray}
\label{eqn:2.8}
\|Kf\|_{V}\lesssim\|f\|_{0}, \quad \|K_\ell f\|_{V}\lesssim\|f\|_{0}\quad \forall f \in L^2(\Omega).
\end{eqnarray}

Let $G_{\ell}: V \rightarrow V_{\ell}$ be the Ritz projection as follows
\begin{align}
\label{eqn:3.6a}
&a(G_{\ell} u,v) = a(u, v) \quad \forall v \in V_{\ell}.
\end{align}
Then $K_\ell = G_\ell \circ K$ and there hold
\begin{eqnarray}
\label{eqn:2.9}
\|K - K_\ell\|_{V}\to 0\quad\text{as } h_{\ell} \to 0\quad\text{and}\quad\|K - K_\ell\|_{0}\to 0\quad\text{as } h_{\ell} \to 0.
\end{eqnarray}

\indent Consider the adjoint eigenvalue problems:
Find $\lambda^{*} \in \mathbb{C}$, $0\not= u^{*} \in V$, such that
\begin{eqnarray}
\label{eqn:2.10}
a(v,u^{*})=\overline{\lambda^{*}} (v,u^{*})_0\quad \forall v \in V.
\end{eqnarray}
Find $\lambda_{\ell}^{*} \in \mathbb{C}$, $0\neq u_{\ell}^{*} \in V_{\ell}$, such that
\begin{eqnarray}
\label{eqn:2.11}
a(v,u_{\ell}^{*})=\overline{\lambda_{\ell}^{*}} (v,u_{\ell}^{*})_0\quad \forall v \in V_{\ell}.
\end{eqnarray}
The source problems associated with (\ref{eqn:2.10}) and (\ref{eqn:2.11}) are as follows, respectively:
Given $g\in L^2(\Omega)$, find $w^{*} \in V$ satisfying
\begin{eqnarray}
\label{eqn:2.12}
a(v,w^{*})=(v,g)_0\quad \forall v \in V.
\end{eqnarray}
Given $g\in L^2(\Omega)$, find $w_{\ell}^{*} \in V_{\ell}$ satisfying
\begin{eqnarray}
\label{eqn:2.13}
a(v,w_{\ell}^{*})=(v,g)_0\quad \forall v \in V_{\ell}.
\end{eqnarray}

\indent Let $K_{*}$ and $K_{\ell*}$ be the solution operators of (\ref{eqn:2.12}) and (\ref{eqn:2.13}), respectively.
Let $G_{\ell*}:H_0^1(\Omega) \rightarrow V_{\ell}$ be the Ritz projection.
Then  $K_{\ell*} = G_{\ell*} \circ K_*$  and it is valid that
\begin{eqnarray}
\label{eqn:2.14}
\|K_{*} - K_{\ell*}\|_{V}\to 0\quad\text{as } h_{\ell} \to 0\quad\text{and}\quad\|K_{*} - K_{\ell*}\|_{0}\to 0\quad\text{as } h_{\ell} \to 0.
\end{eqnarray}

(\ref{eqn:2.9})  and (\ref{eqn:2.14}) show that the error estimate results in spectral approximation theory hold when $h$ is sufficiently small.

From \cite{grisvard2011,bernardi2000}, we have the following regularity estimate for the source problem (\ref{eqn:2.4}) and the same result for the adjoint problem (\ref{eqn:2.12}) also holds. There exists a constant $r>0$ such that
\begin{equation}\label{eqn: regularity}
\| w \|_{1+r} \lesssim \|f\|_0 \quad\forall f\in L^2(\Omega).
\end{equation}

\indent Let $\{\lambda_{j}\}$ and $\{\lambda_{j,\ell}\}$ be enumerations of the eigenvalues of (\ref{eqn:2.3}) and (\ref{eqn:2.5}) according to their modulus from small to large with each repeated according to their multiplicities. 
Denote the index set ${J}=\{n+1,\ldots,n+N\}$.
Let $\widehat{\lambda} = \frac{1}{N}\sum_{j\in {J}}\lambda_{j}$ and $\widehat{\lambda}_{\ell}=\frac{1}{N}\sum_{j\in {J}}\lambda_{j,\ell}$. 
Let $\{\lambda_{j}\}_{j\in{J}}$ be an eigenvalue cluster, i.e., $\lambda_n \neq \lambda_{n+1}, \lambda_{n+2}, ...  ,\lambda_{n+N}\neq \lambda_{n+N+1}$, and denote the set of reciprocals of these $N$ eigenvalues by $\sigma$. 
We assume the separation
$M_J := \sup _{\mathcal{T}_\ell \in \mathbb{T}} \max_{j \in \{1, \ldots, \operatorname{dim} V_\ell \} \backslash J} \max_{k \in J} \frac{|\lambda_k|}{|\lambda_{j,\ell}-\lambda_k|} < \infty$.
Let $\Gamma$ be a Jordan closed curve and $\Gamma\subset \rho(K)$, and let $\sigma$ be enclosed by $\Gamma$ and $\Gamma$ enclose no other eigenvalues of $K$.  
Define the spectral projection $ E=\frac{-1}{2\pi\mathrm{i}}\int_{\Gamma}(z-K)^{-1}\mathrm{d}z$ associated with $K$ and $\sigma$, and define $E_{\ell}=\frac{-1}{2\pi\mathrm{i}}\int_{\Gamma}(z-K_\ell)^{-1}\mathrm{d}z$ associated with $K_\ell$ and the eigenvalues of $K_\ell$ which lie in $\Gamma$.\\
\indent We assume that the initial mesh size $h_0$ is sufficiently small such that $\dim R(E) =\dim R(E_{\ell}) $ and $\dim R(E_{*}) =\dim R(E_{\ell*})$ (c.f. page 684 in \cite{babuska1991}), where $R$ denotes the range.
Denote that $\widehat{R}(E)=\{v\in R(E): \|v\|_{V}=1\}$ and $\widehat{R}(E_{\ell})=\{v\in R(E_{\ell}): \|v\|_{V}=1\}$, and denote $\widehat{R}(E_{*})=\{v\in R(E_{*}): \|v\|_{V}=1\}$, and $\widehat{R}(E_{\ell*})=\{v\in R(E_{\ell*}): \|v\|_{V}=1\}$.
We call $R(E)$ and $R(E_{*})$ as eigenspaces, and call $R(E_{\ell})$ and $R(E_{\ell*})$ as approximate eigenspaces.\\
\indent For two closed subspaces $A$ and $B$ of $V$, we define the gap $\widehat{\delta}(A,B)$ between $A$ and $B$ in the sense of norm $\|\cdot\|_{V}$ as
\begin{equation*}
 \delta(A,B) = \sup_{\phi\in A,\|\phi\|_{V}=1}\inf_{\psi\in B}\|\phi-\psi\|_{V},\quad
\widehat{\delta}(A,B) = \max\{\delta(A,B),\delta(B,A)\}.
\end{equation*}
\indent Define 
\begin{align*}
&\varepsilon_{\ell} = \sup\limits_{u\in \widehat{R}(E)}\inf\limits_{\chi\in V_{\ell}}\|u-\chi\|_{V},\quad
\varepsilon_{\ell}^{*} = \sup\limits_{u\in \widehat{R}(E_{*})}\inf\limits_{\chi\in V_{\ell}}\|u-\chi\|_{V}.
\end{align*}

\indent We assume that the following conditions (C1) or (C2) hold.\\
\indent (C1). $\{u_{j,\ell}\}_{j\in J}$ and $\{u^*_{j,\ell}\}_{j\in J}$ are an orthonormal basis in $R(E_\ell)$ and $R(E_{\ell*})$ with respect to $(\cdot,\cdot)_0$, respectively.

\indent (C2). $\|u_{j,\ell}\|_0 = 1$, $\|u_{j,\ell}^\ast\|_0 = 1$, and $(u_{i,\ell}, u_{j,\ell}^\ast)_0 = 0$ ($i \neq j$), and there exists a positive constant $\tau$ independent of $h$, such that $|(u_{j,\ell}, u_{j,\ell}^\ast)_0| \ge \tau$ for all $j \in J$.

\begin{remark}
When we obtain $\{u_{j,\ell}\}_{j\in J}\subset R(E_\ell)$ and
$\{u^*_{j,\ell}\}_{j\in J}\subset R(E_{\ell*})$ by solving (\ref{eqn:2.5}) and (\ref{eqn:2.10}), respectively,
if the condition (C1) does not hold for  $\{{u}_{j,\ell}\}_{j\in J}$
and $\{u^*_{j,\ell}\}_{j\in J}$, the Gram--Schmidt process changes 
$\{u_{j,\ell}\}_{j\in J}$ and $\{u^*_{j,\ell}\}_{j\in J}$ 
to orthonormal bases (still denoted by)
 $\{u_{j,\ell}\}_{j\in J}$ and $\{u^*_{j,\ell}\}_{j\in J}$
in $R(E_\ell)$ and $R(E_{\ell*})$ with respect to $(\cdot,\cdot)_0$, respectively. Then $\{u_{j,\ell}\}_{j\in J}$ and $\{u^*_{j,\ell}\}_{j\in J}$ satisfy condition (C1).

When the eigenvalue cluster $\{\lambda_{j,\ell}\}_{j\in J}$ is composed of $N$ simple eigenvalues, then using the arguments of Lemma~5 in \cite{Wang2025posteriori}, we can deduce that the corresponding eigenfunctions $\{u_{j,\ell}\}_{j\in J}$ and $\{u^*_{j,\ell}\}_{j\in J}$ satisfy condition (C2).
\end{remark}

From Theorem 3.5 of \cite{yang2024}, we have the following lemma. 

\begin{lemma}\label{lem:2.3}
Assume that the condition (C1) or (C2) is valid. Then when the initial mesh size $h_0$ is small enough, there hold
\begin{align}
\label{eqn:2.20}
&\widehat{\delta}(R(E),R(E_{\ell}))\lesssim \sum\limits_{j\in{J}}\|({K} - K_\ell)u_{j,\ell}\|_{V} \lesssim N\widehat{\delta}(R(E),R(E_{\ell})),\\
\label{eqn:2.21}
&\widehat{\delta}(R(E_{*}),R(E_{\ell*}))\lesssim \sum\limits_{j\in{J}}\|(K_{*} - K_{\ell*})u_{j,\ell}^{*}\|_{V} \lesssim N\widehat{\delta}(R(E_{*}),R(E_{\ell*})),\\
\label{eqn:2.22}
&|\widehat{\lambda}-\widehat{\lambda}_{\ell}| \lesssim N^2 \varepsilon_\ell\varepsilon_\ell^*\lesssim N^2\sum\limits_{j\in{J}}\left(\|({K} - K_\ell)u_{j,\ell}\|_{V}^{2} + \|(K_{*} - K_{\ell*})u_{j,\ell}^{*}\|_{V}^{2}\right).
\end{align}
When the eigenvalue cluster $\{\lambda_{j}\}_{j\in J}$ is composed of non-defective eigenvalues, i.e., eigenvalues for which the algebraic multiplicity equals the geometric multiplicity, then for $k\in J$,
\begin{equation}
\label{eqn:2.22a}
|\lambda_{k}-\lambda_{k,\ell}|\lesssim N^2 \varepsilon_\ell\varepsilon_\ell^*\lesssim N^2\sum\limits_{j\in{J}}(\|({K} - K_\ell)u_{j,\ell}\|_{V}^{2} + \|(K_{*} - K_{\ell*})u_{j,\ell}^{*}\|_{V}^{2}).
\end{equation} 
\end{lemma}

\begin{remark}
Note that $K u_{j,\ell}$ and $K_{\ell} u_{j,\ell}$ are the solutions of equations (\ref{eqn:2.5}) and (\ref{eqn:2.10}), respectively, with the right-hand side term $f = u_{j,\ell}$. Lemma 2.2 reveals that the a posteriori errors of the approximate eigenvalue and approximate eigenspace can be derived from the a posteriori error of the approximate solution to the associated source problem. To compute the set $\{K_\ell u_{j,\ell}\}_{j \in J}$, it is necessary to solve the linear system (\ref{eqn:2.10}) for each $f = u_{j,\ell}$ ($j \in J$). Notably, these $q$ linear systems share an same coefficient matrix but differ in their right-hand sides.

In most cases, the eigenvalue clusters $\{\lambda_{j}\}_{j\in J}$ will be split into $N$ simple eigenvalues $\{\lambda_{j,\ell}\}_{j\in J}$ under discretization. In this case,  let $u_{j,\ell}$ denote the eigenfunction associated with $\lambda_{j,\ell}$ ($j \in J$). Then there holds
\begin{equation}
K_{\ell} u_{j,\ell} = \lambda_{j,\ell}^{-1} u_{j,\ell},
\end{equation}
which is precisely the solution to the linear system (\ref{eqn:2.10}) with $u_{j,\ell}$ as the right-hand side.
\end{remark}

We split $\mathcal{E}_\ell := \mathcal{E}^0_\ell \cup\mathcal{E}^{\partial}_\ell$,
where $\mathcal{E}^0_\ell$ and $\mathcal{E}^{\partial}_\ell$ denote the set of internal faces and the set of boundary faces of partition $\mathcal T_\ell$, respectively.
Let
\begin{equation*}
[\![ \boldsymbol{A} \nabla w \cdot \nu ]\!]|_e = \boldsymbol{A} \nabla w|_{T^+} \cdot \nu^{+}+\boldsymbol{A} \nabla w|_{T^-} \cdot \nu^{-},
\end{equation*}
where $e$ is the common face of elements $T^{+}$and $T^{-}$with unit outward normals $\nu^{+}$and $\nu^{-}$, respectively, and $\nu=\nu^{-}$. \\
\indent Referring to the work of Feischl~et~al.~\cite{feischl2014}, the local error estimator for the source problem (\ref{eqn:2.6}) is  
\begin{align*}
    & \eta_\ell(f, T)^2 :=  h_T^2 \|R_T (f, w_\ell )\|_{0, T}^2 + \sum_{e \in \mathcal{E}_{\ell}, e \subset \partial T} h_e \|J_e (w_\ell) \|_{0, e}^2,
\end{align*}
where the element residual $R_T(f, w_\ell)$ and the jump residual $J_e(w_\ell)$ are defined as follows
\begin{align*}
    R_T (f, w_\ell) & = f - \mathcal{L} w_\ell \\
    & =f + \nabla \cdot (\boldsymbol{A} \nabla w_\ell) - \boldsymbol{b} \cdot \nabla w_\ell - c w_\ell \quad \text { for } T \in \mathcal T_\ell, \\
    J_e(w_\ell)= & 
    \begin{cases}
        [\![ \boldsymbol{A} \nabla  w_\ell \cdot \nu ]\!] & \text { if } e \in \mathcal{E}^0_\ell, \\
        0 & \text { if } e \in \mathcal{E}^{\partial}_\ell.
    \end{cases}
\end{align*}
And the local error estimator for the source problem (\ref{eqn:2.6}) is defined as 
\begin{align*}
    & \eta^*_\ell(g, T)^2 :=  h_T^2 \|R^*_T (g, w^*_\ell )\|_{0, T}^2 
    + \sum_{e \in \mathcal E_\ell, e \subset \partial T} h_e \|J^*_e ( w^*_\ell) \|_{0, e}^2,
\end{align*}
where the element residual $R_T^*(g, w_\ell^*)$ and the jump residual $J_e^*(w_\ell^*)$ are defined as follows
\begin{align*}
    R_T^*(g, w^*_\ell) 
    & = g - \mathcal{L}^* w^*_\ell \\
    &= g + \nabla \cdot(\overline{\boldsymbol{A}} \nabla w^*_\ell) + \nabla \cdot(\overline{\boldsymbol{b}} w^*_\ell) - \overline{c} w^*_\ell \\
    & = g +\nabla \cdot(\overline{\boldsymbol{A}} \nabla w^*_\ell) + \overline{\boldsymbol{b}} \cdot \nabla w^*_\ell + (\nabla \cdot \overline{\boldsymbol{b}} - \overline{c}) w^*_\ell \quad \text { for } T \in \mathcal{T}_\ell, \\
    J_e^*(w_\ell^*) = & 
    \begin{cases}
        [\![ \overline{\boldsymbol{A}} \nabla w^*_\ell \cdot \nu ]\!] & \text { if } e \in \mathcal{E}^0_{\ell}, \\
        0 & \text { if } e \in \mathcal{E}^{\partial}_{\ell}.
    \end{cases}
\end{align*}

For a given $\mathcal M \subseteq \mathcal T_\ell$, we define the error estimators $\eta_\ell(f, \mathcal M)$ and $\eta^*_\ell(g, \mathcal M)$ as 
\begin{align*}
\eta_\ell(f, \mathcal M)^2 := \sum_{T\in\mathcal M}\eta_\ell(f, T)^2 \text{ and } \eta^*_\ell(g, \mathcal M)^2 := \sum_{T\in\mathcal M}\eta^*_\ell(g, T)^2.
\end{align*}
As shown in \cite{ainsworth2000,verfurth2013} and Chapter 34 of \cite{ern2021}, the error estimators are reliable and efficient; see also \cite{feischl2014}.
\begin{lemma}\label{lem:2.2}
Assume that the initial mesh size $h_0$ is small enough. Then for all regular triangulations $\mathcal{T}_\ell$ and corresponding solutions $w_\ell$ and $w^*_\ell$ of \eqref{eqn:2.6} and \eqref{eqn:2.13}, there hold
\begin{align}
\label{eqn:2.24a}
&\|w - w_\ell\|_V \lesssim \eta_\ell(f, \mathcal T_\ell),\\
\label{eqn:2.24b}
&\|w^* - w^*_\ell\|_V \lesssim \eta^*_\ell(g, \mathcal T_\ell),\\
\label{eqn:2.24c}
 &\eta_\ell(f, \mathcal T_\ell) \lesssim \|w - w_\ell\|_V + \mathrm{osc}_\ell(f, w_\ell),\\
 \label{eqn:2.24d}
 &\eta^*_\ell(g, \mathcal T_\ell) \lesssim \|w^* - w^*_\ell\|_V + \mathrm{osc}^*_\ell(g, w_\ell^*),
\end{align}
where oscillation terms
\begin{align*}
\mathrm{osc}_\ell(f, w_\ell)^2 := & \sum_{T\in\mathcal T_\ell}\bigg( h^2_T \|(1-\Pi_\ell^{2p-1})(\mathcal{L}(w_\ell)-f)\|^2_{0,T}\\ 
&\quad +\sum_{e \in \mathcal{E}^0_{\ell}, e \subset \partial T} h_e \|(1-\Pi_\ell^{2p-1})[\![ \boldsymbol{A} \nabla w_\ell \cdot \nu  ]\!]\|^2_{0,e} \bigg),\\
\mathrm{osc}^*_\ell(g, w_\ell^*)^2 := & \sum_{T\in\mathcal T_\ell}\bigg( h^2_T \|(1-\Pi_\ell^{2p-1})(\mathcal{L}^*(w_\ell^*)-g)\|^2_{0,T}\\ 
&\quad +\sum_{e \in \mathcal{E}^0_{\ell}, e \subset \partial T} h_e \|(1-\Pi_\ell^{2p-1})[\![ \overline{\boldsymbol{A}} \nabla w_\ell^* \cdot \nu  ]\!]\|^2_{0,e} \bigg),
\end{align*}  
with the $L^2$-orthogonal projection $\Pi_\ell^{q}:\,L^2(\Omega)\to\mathcal{P}^{q}(\mathcal T_\ell)$.
\end{lemma}

\section{Computable and theoretical error estimators}\label{sec3}
Lemma \ref{lem:2.3} tells us that the error estimates of approximate eigenspace and approximate eigenvalues are attributed to the error estimate of finite element solutions of the source problems (\ref{eqn:2.4}) and (\ref{eqn:2.12}) with right-hand sides $f=u_{j,\ell}$ and $g=u_{j,\ell}^*$, respectively.
Therefore we give the following computable error estimators.\\
\indent For $T \in \mathcal{T}_\ell$ and $j\in{J}$, we define the computable local error estimator on $T$ as $\eta_{\ell}(u_{j,\ell},T)$ and $\eta^*_{\ell}(u^*_{j,\ell},T)$.
For a given $\mathcal M \subseteq \mathcal T_\ell$, we define the computable error estimators as $\eta_{\ell}(u_{j,\ell}, \mathcal M)$ and $\eta^*_{\ell}(u^*_{j,\ell}, \mathcal M)$.\\
\indent Using the computable error estimators in this paper and consulting the existing standard algorithms, we present the following algorithm.\\
\indent {\bf Algorithm 1.} 
\textsc{Input:} Initial triangulation $\mathcal{T}_{0}$ with mesh size $h_0$ and the bulk parameter $\theta\in(0,1].$ \\
\textsc{Adaptive loop:} For all $\ell=0,1,2,\ldots$ iterate Steps~{\rm{(i)--(v)}} as follows.
\begin{itemize}
\item[\rm{(i)}] Solve (\ref{eqn:2.5}) and (\ref{eqn:2.10}) on $\mathcal{T}_{\ell}$ with mesh size $h_{\ell}$ for discrete solution \\ $(\lambda_{j,\ell}, u_{j,\ell},u^*_{j,\ell})$ with $\|u_{j,\ell}\|_{V}=1$ and  $\|u^*_{j,\ell}\|_{V}=1$ for $j\in{J}$.
\item[\rm{(ii)}] Compute the corresponding estimators $\eta_{\ell}(u_{j,\ell}, T)$ and $\eta^*_{\ell}(u^*_{j,\ell}, T)$ for all $T\in\mathcal{T}_{\ell}$.
\item[\rm{(iii)}] Determine a subset $\mathcal{M}_{\ell}$ of $\mathcal{T}_{\ell}$ of minimal cardinality by {\bf Marking Strategy} such that 
\begin{equation}\label{algorithm}
\sum_{j\in J}(\eta_{\ell}(u_{j,\ell}, \mathcal{M}_{\ell})^2 + \eta^*_{\ell}(u^*_{j,\ell}, \mathcal{M}_{\ell})^2)
\geq
\theta
\sum_{j\in J}( \eta_{\ell}(u_{j,\ell}, \mathcal T_{\ell})^2 + \eta^*_{\ell}(u^*_{j,\ell},\mathcal T_{\ell})^2).
\end{equation}
\item[\rm{(iv)}] Generate $\mathcal{T}_{\ell+1}:={\rm Refine}(\mathcal{T}_{\ell},\mathcal{M}_{\ell})$ by the NVB refinement strategy.
\item[\rm{(v)}] Increase $\ell \rightarrow \ell+1$ and goto Step~\rm{(i)}.
\end{itemize}
\textsc{Output:} Approximate eigenpairs $(\lambda_{j,\ell}, u_{j,\ell},u^*_{j,\ell})$ and error estimators $\eta_{\ell}(u_{j,\ell}, T)$ and $\eta^*_{\ell}(u_{j,\ell}, T)$ for all $\ell\in\mathbb{N}$.\\

\indent The marking strategy in Step (iii) was introduced by D$\ddot{\mathrm{o}}$rfler~\cite{dorfler1996}.\\
\indent In this paper, we assume that $\{u_{j}\}_{j\in {J}}$ and $\{u^*_{j}\}_{j\in{J}}$ are any given orthonormal basis with respect to $(\cdot,\cdot)_0$ in $R(E)$ and $R(E_{*})$, respectively.

\begin{theorem}\label{thm:3.1}
Assume that the initial mesh size $h_0$ is small enough, then the following estimates hold
\begin{align}
    \label{eqn:3.1}
    &\|(K - K_\ell)|_{R(E)}\|_{V}\lesssim \sum_{j\in{J}}\|(K - K_\ell)u_{j}\|_{V}\lesssim N\|(K - K_\ell)|_{R(E)}\|_{V},\\
    \label{eqn:3.2}
    &\|(K_{*}-K_{\ell*})|_{R(E_{*})}\|_{V}\lesssim\sum_{j\in{J}}\|(K_{*} - K_{\ell*})u_{j}^{*}\|_{V}\lesssim N\|(K_{*}-K_{\ell*})|_{R(E_{*})}\|_{V},\\
    \label{eqn:3.3}
    &\widehat{\delta}(R(E),R(E_{\ell}))\lesssim \sum_{j\in{J}}\|({K} - K_\ell)u_{j}\|_{V}\lesssim N\widehat{\delta}(R(E),R(E_{\ell})),\\
    \label{eqn:3.4}
    &\widehat{\delta}(R(E_{*}),R(E_{\ell*}))\lesssim \sum_{j\in{J}}\|(K_{*} - K_{\ell*})u_{j}^{*}\|_{V}\lesssim N\widehat{\delta}(R(E_{*}),R(E_{\ell*})),\\
    \label{eqn:3.5}
    &|\widehat{\lambda}-\widehat{\lambda}_{\ell}| \lesssim N^2\sum_{j\in{J}}\left(\|({K} - K_\ell)u_{j}\|_{V}^{2} + \|(K_{*} - K_{\ell*})u_{j}^{*}\|_{V}^{2}\right).
\end{align}
When the eigenvalue cluster $\{\lambda_{j}\}_{j\in J}$ is composed of non-defective eigenvalues, then for $k\in J$,
\begin{equation}
\label{eqn:3.5a}
|\lambda_{k}-\lambda_{k,\ell}|\lesssim N^2\sum\limits_{j\in{J}}\left(\|({K} - K_\ell)u_{j}\|_{V}^{2} + \|(K_{*} - K_{\ell*})u_{j}^{*}\|_{V}^{2}\right).
\end{equation} 
\end{theorem}
\begin{proof}
Thanks to $u = \sum_{j\in{J}} (u,u_j)_0u_j$,  we deduce
\begin{eqnarray*}
&&\|(K - K_{\ell})|_{R(E)}\|_{V}=\sup_{u\in \widehat{R}(E)}\|(K - K_{\ell})u\|_{V}\\
&&\quad =\sup_{u\in \widehat{R}(E)}\|(K - K_{\ell})\sum\limits_{j\in{J}} (u,u_j)_0u_j\|_{V}\\
&&\quad \leq \sup_{u\in \widehat{R}(E)}\sum\limits_{j\in{J}}|(u,u_j)_0|\|(K - K_{\ell})u_j\|_{V}\\
&&\quad \lesssim \sum_{j\in{J}}\|(K - K_{\ell})u_{j}\|_{V}.
\end{eqnarray*}
We obtain the first estimate of (\ref{eqn:3.1}). 
And the second estimate of (\ref{eqn:3.1}) is obvious. Similarly, we can deduce (\ref{eqn:3.2}).\\
\indent Referring to Lemma 3.4 of \cite{yang2024}, we have
\begin{equation*}
    \widehat{\delta}(R(E),R(E_{\ell}))\simeq \|(K - K_\ell)|_{R(E)}\|_{V}.
\end{equation*} 
Then combining \eqref{eqn:3.1}, we obtain \eqref{eqn:3.3}. Similarly, we can deduce (\ref{eqn:3.4}).\\
\indent From Theorem 2.2 of \cite{yang2024}, (\ref{eqn:3.3}) and  (\ref{eqn:3.4}), we have 
$$\varepsilon_\ell \lesssim \widehat{\delta}(R(E),R(E_{\ell}))\lesssim \sum_{j\in{J}}\|({K} - K_\ell)u_{j}\|_{V},$$ 
$$\varepsilon_\ell^*\lesssim \widehat{\delta}(R(E_*),R(E_{\ell*}))\lesssim \sum_{j\in{J}}\|({K}_* - K_{\ell*})u^*_{j}\|_{V}.$$
Then combining \eqref{eqn:2.22} we obtain \eqref{eqn:3.5} and combining \eqref{eqn:2.22a} we obtain \eqref{eqn:3.5a}.
\end{proof}

Theorem \ref{thm:3.1} shows that the error estimates of approximate eigenspace and approximate eigenvalues are attributed to the error estimate of finite element solutions of the source problems (\ref{eqn:2.4}) and (\ref{eqn:2.12}) with right-hand sides $f=u_j$ and $g=u_j^*$, respectively. 
Therefore we give the following theoretical error estimators.\\
\indent For $T \in \mathcal{T}_\ell$ and $j\in{J}$, due to \eqref{eqn:2.20} and \eqref{eqn:3.3}, and 
\eqref{eqn:2.21} and \eqref{eqn:3.4}, we define the theoretical local error estimator on $T$ as
$\mu_\ell(u_j, T)^2 :=\eta_\ell(u_j, T)^2$ and $ \mu^*_\ell(u^*_j, T)^2 :=\eta^*_\ell(u^*_j, T)^2$.
Referring to \cite{gallistl2015,bonito2016}, we specify the notations $\mu_\ell$ and $\mu^*_\ell$ as the theoretical error estimators in this paper.
For a given $\mathcal M \subseteq \mathcal T_\ell$, we define the theoretical error estimators 
on $\mathcal M$ as 
\begin{align*} 
\mu_\ell(u_j, \mathcal M)^2 := \sum_{T\in\mathcal M}\mu_{\ell}(u_j, T)^2 \text{ and } \mu^*_{\ell}(u^*_j,\mathcal M)^2 := \sum_{T\in\mathcal M}\mu^*_{\ell}(u^*_j,T)^2.
\end{align*}

\indent Consider (\ref{eqn:2.24a}), (\ref{eqn:2.24b}), (\ref{eqn:2.24c}), and (\ref{eqn:2.24d}) with the right-hand sides $f=u_j$ and $g=u^*_j$, respectively.
Then the oscillation terms are
\begin{align*}
\mathrm{osc}_\ell(u_j, K_\ell u_j)^2 := & \sum_{T\in\mathcal T_\ell}\bigg( h^2_T \|(1-\Pi_\ell^{2p-1})(\mathcal{L}(K_\ell u_j)-u_j)\|^2_{0,T}\\ 
&\quad +\sum_{e \in \mathcal{E}^0_{\ell}, e \subset \partial T} h_e \|(1-\Pi_\ell^{2p-1})[\![ \boldsymbol{A} \nabla (K_\ell u_j) \cdot \nu  ]\!]\|^2_{0,e} \bigg),\\
\mathrm{osc}^*_\ell(u^*_j, K^*_\ell u^*_j)^2 := & \sum_{T\in\mathcal T_\ell}\bigg( h^2_T \|(1-\Pi_\ell^{2p-1})(\mathcal{L}^*(K^*_\ell u^*_j) - u^*_j)\|^2_{0,T}\\ 
&\quad +\sum_{e \in \mathcal{E}^0_{\ell}, e \subset \partial T} h_e \|(1-\Pi_\ell^{2p-1})[\![ \overline{\boldsymbol{A}} \nabla (K^*_\ell u^*_j) \cdot \nu  ]\!]\|^2_{0,e} \bigg).
\end{align*} 
It can be easily verified that
when the initial mesh size of $h_0$ is sufficiently small,
\begin{equation*}
\mathrm{osc}_\ell(u_j, K_\ell u_j)^2 \leq \frac{1}{2} \sum_{j\in J} \mu_\ell(u_j,\mathcal{T}_\ell)^2,\quad \mathrm{osc}^*_\ell(u^*_j, K^*_\ell u^*_j)^2 \leq \frac{1}{2} \sum_{j\in J} \mu^*_\ell(u^*_j,\mathcal{T}_\ell)^2,
\end{equation*}
since the coefficients $\boldsymbol{A}$, $\boldsymbol{b}$ and $c$ are piecewise sufficiently smooth functions.
Thus when the initial mesh size of $h_0$ is sufficiently small, from Lemma \ref{lem:2.2},
there exist ${C_{\mathrm{rel}}}$ and $C_{\mathrm{eff}}$, independent of $h_\ell$, such that
\begin{align}\label{eqn:rel_a}
&\sum_{j\in J} \|Ku_j - K_\ell u_j\|_V^2
\leq
{C^2_{\mathrm{rel}}}
\sum_{j\in J} \mu_\ell(u_j,\mathcal{T}_\ell)^2,\\ \label{eqn:rel_b}
&\sum_{j\in J} \|K_*u^*_j - K_{\ell*} u^*_j\|_V^2
\leq
{C^2_{\mathrm{rel}}}
\sum_{j\in J} \mu^*_\ell(u^*_j,\mathcal{T}_\ell)^2,\\
\label{eqn:eff_a}
&C_{\mathrm{eff}}^{-2}\sum_{j\in J}\mu_\ell(u_j,\mathcal{T}_\ell)^2 \le\sum_{j\in J}\|Ku_j - K_\ell u_j\|^2_V ,\\ 
\label{eqn:eff_b}
&C_{\mathrm{eff}}^{-2}\sum_{j\in J}\mu^*_\ell(u_j^*,\mathcal{T}_\ell)^2 \le\sum_{j\in J}\|K_*u_j^* - K_{\ell*} u_j^*\|^2_V.
\end{align} 

\indent Let $\Lambda_{\ell} := E_{\ell} \circ G_{\ell}$ and $\Lambda_{\ell*} := E_{\ell*} \circ G_{\ell*}$. 

Thanks to Proposition 5.1 of \cite{gallistl2015} and Lemma 1 of \cite{bonito2016}, we deduce the theorem below.

\begin{theorem}\label{thm:3.2}
Assume that the initial mesh size $h_0$ is small enough so that 
\begin{align}
\label{eqn:3.7}
&\max_{j \in  J} \|u_j- \Lambda_{\ell} u_j\|_0 \le \sqrt{1+(2N)^{-1}} -1, \\
\label{eqn:3.7_adjoint}
&\max_{j \in  J} \|u^*_j- \Lambda_{\ell*} u^*_j\|_0 \le \sqrt{1+(2N)^{-1}} -1.
\end{align}
Then for $T \in \mathcal T_{\ell}$,
\begin{align}
\label{eqn:3.8}
&\sum_{j\in J} \mu_{\ell}(u_j, T)^2 \le \frac{303}{200}\sum_{j\in J}\eta_{\ell}(u_{j,\ell}, T)^2 + 101\vartheta_{\ell}(T)^2, \\
\label{eqn:3.9}
&\sum_{j\in J} \eta_{\ell}(u_{j,\ell}, T)^2 \le \frac{101}{50}\sum_{j\in J}\mu_{\ell}(u_j, T)^2 + 202\vartheta_{\ell}(T)^2,\\
\label{eqn:3.8_adjoint}
&\sum_{j\in J} \mu^*_{\ell}(u_j, T)^2 \le \frac{303}{200}\sum_{j\in J}\eta^*_{\ell}(u^*_{j,\ell}, T)^2 + 101\vartheta^*_{\ell}(T)^2, \\
\label{eqn:3.9_adjoint}
&\sum_{j\in J} \eta^*_{\ell}(u^*_{j,\ell}, T)^2 \le \frac{101}{50}\sum_{j\in J}\mu^*_{\ell}(u^*_j, T)^2 + 202\vartheta^*_{\ell}(T)^2,
\end{align}
where $\vartheta_{\ell}(T)^2:= \sum_{j\in J}(h_T^2\|R_{1,j}\|^2_{0,T} + \sum_{e \in \mathcal{E}_{\ell}, e \subset \partial T}h_e\|R_{2,j}\|^2_{0,e})$ with $R_{1,j} := (u_j - \Lambda_{\ell} u_j) - \mathcal{L}(K_{\ell}(u_j - \Lambda_{\ell} u_j))$ and $R_{2,j} := [\![ \boldsymbol{A} \nabla (K_{\ell}(u_j - \Lambda_{\ell} u_j))\cdot \nu ]\!]$, analogously we can define $\vartheta^*_{\ell}(T)$.
Furthermore we have
\begin{equation}\label{eqn:3.10}
 \sum_{j\in J}
( \mu_{\ell}(u_j, \mathcal M_{\ell})^2 + \mu^*_{\ell}(u_j,\mathcal M_{\ell})^2 ) 
\geq
\tilde{\theta}
 \sum_{j\in J}( \mu_{\ell}(u_j,\mathcal T_{\ell})^2 + \mu^*_{\ell}(u_j,\mathcal T_{\ell})^2),
\end{equation}
for the modified bulk parameter
\begin{equation}\label{eqn:3.11}
\tilde{\theta} = \left(\frac{50}{101}\times\frac{200}{303} - CNh_{\ell}^{2r}\right)\theta \geq \frac{1}{4}\theta >0.
\end{equation}
\end{theorem}

\begin{proof}
Using \eqref{eqn:2.4}, \eqref{eqn:2.6}, and \eqref{eqn:3.6a}, we compute
\begin{align}
\label{eqn:3.12}
\begin{aligned}
u_j&=  (u_j - \Lambda_{\ell} u_j) + \Lambda_{\ell} u_j \\ 
   & = (u_j - \Lambda_{\ell} u_j) + \sum_{m \in J} (\Lambda_{\ell} u_j, u_{m,\ell})_0 u_{m,\ell}.
\end{aligned}
\end{align}
Then
\begin{align}
\label{eqn:3.13}
\begin{aligned}
K_{\ell}u_j &= K_{\ell}(u_j - \Lambda_{\ell} u_j) + \sum_{m \in J} (\Lambda_{\ell} u_j, u_{m,\ell})_0 K_{\ell} u_{m,\ell}.
\end{aligned}
\end{align}

Combining the above two equations yields
\begin{align}\label{eqn:3.14}
u_j - \mathcal{L} ({K}_{\ell} u_j) = & \sum_{m \in J} (\Lambda_{\ell} u_j, u_{m,\ell})_0 (u_{m,\ell} - \mathcal{L} ({K}_{\ell} u_{m,\ell})) + R_{1,j},\\ \label{eqn:3.15}
[\![ \boldsymbol{A} \nabla ({K}_{\ell} u_j)\cdot \nu ]\!] = & \sum_{m \in J} (\Lambda_{\ell} u_j, u_{m,\ell})_0 [\![ \boldsymbol{A} \nabla  ({K}_{\ell} u_{m,\ell}) \cdot \nu ]\!] + R_{2,j}.
\end{align}

We now define the following vectors 
\begin{align*}
U & := [u_j - \mathcal{L}({K}_{\ell} u_j], \quad U_{\ell}  := [u_{m,\ell} - \mathcal{L}({K}_{\ell} u_{m,\ell})], \\
W & := [[\![ \boldsymbol{A} \nabla ({K}_{\ell} u_j) \cdot \nu ]\!] ],\quad W_{\ell}:= [[\![ \boldsymbol{A} \nabla ({K}_{\ell} u_{m,\ell})\cdot \nu ]\!] ],\\
R_1 & := [R_{1,j}], \quad R_2 := [R_{2,j}],
\end{align*}
where \( n+1 \leq j, m \leq n+N \). 
We also define the matrix $M\in \mathbb{C}^{N\times N}$ with entries $M_{j,m}= (\Lambda_{\ell} u_j, u_{m,\ell})_0$ with \( j \) representing the row index and \( m \) the column index.
The relationships given by \eqref{eqn:3.14} and \eqref{eqn:3.15} in the $L^2$ sense can be expressed as
\begin{align}
\label{eqn:3.16}
U = MU_{\ell} + R_1, \quad W = M W_{\ell} + R_2.
\end{align}
We let \( \|\cdot\|_2 \) denote the operator norm in \( {\ell}_2(\mathbb{C}^N) \) and also let \( \|\cdot\|_2 \) denote the Euclidean length. 

Let \( v \in \mathbb{C}^N \). We have
\(
\|Mv\|_2^2 = v^{\mathsf{H}} {M}^{\mathsf{H}} M v \leq \|{M}^{\mathsf{H}} M\|_2 \|v\|_2^2.
\)
Thus, we can derive the inequality
\(
\|M\|_2^2 \leq \|{M}^{\mathsf{H}} M\|_2.
\)
Since \( {M}^{\mathsf{H}} M \) is self-adjoint and positive semi-definite, it follows that 
\(
\|{M}^{\mathsf{H}} M\|_2
\)
is equal to the maximum eigenvalue of \( {M}^{\mathsf{H}} M \).

The row $j$ and the column $m$ element of matrix $M{M}^{\mathsf{H}}$ is
\begin{equation}
\label{eqn:3.18}
\begin{aligned}
\sum_{i=1}^N (\Lambda_{\ell} u_j, u_{i,\ell})_0\overline{(\Lambda_{\ell} u_m, u_{i,\ell})_0}
&= \sum_{i=1}^N (\Lambda_{\ell} u_j, (\Lambda_{\ell} u_m, u_{i,\ell})_0u_{i,\ell})_0\\
&= \bigg(\Lambda_{\ell} u_j, \sum_{i=1}^N(\Lambda_{\ell} u_m, u_{i,\ell})_0u_{i,\ell}\bigg)_0\\
&= (\Lambda_{\ell} u_j, \Lambda_{\ell}u_{m})_0.
\end{aligned}
\end{equation}
Thus, we define \(B := M M^{\mathsf{H}} = [(\Lambda_{\ell} u_j, \Lambda_{\ell} u_m)_0]\). 
A fundamental result in linear algebra indicates that \(B\) is isospectral with \(M^{\mathsf{H}} M\). Consequently, it follows that \(\|M^{\mathsf{H}} M\|_2 = \|B\|_2\), and both quantities correspond to the maximum eigenvalue of \(B\). The matrix \(B\) is thoroughly analyzed in the proof of Lemma 5.1 of \cite{gallistl2015}. Notably, \(B\) is nonsingular under the condition \eqref{eqn:3.7}, and according to (5.2) and the subsequent results in \cite{gallistl2015}, we obtain
\begin{equation}
\label{eqn:3.19}
\frac{2N-1}{2N}  \le B_{ii} \le \frac{2N+1}{2N} \qquad  \text{and} \qquad  \sum_{j \neq i} |B_{ij}|  \le \frac{N-1}{2N}.
\end{equation}
Gershgorin's theorem implies that the eigenvalues \(\{\sigma_i\}\) of the matrix \(B\) are constrained by the following relationship
\begin{align}
\label{eqn:3.20}
1 \le 2\sigma_i \le 3, ~1 \le i \le N.
\end{align}
Thus $\|M\|_2^2 \le \|{M}^{\mathsf{H}} M\|_2 = \|M {M}^{\mathsf{H}}\|_2\le 3/2$, and by the Young inequality with $\epsilon=10^{-2}$, 
\begin{align*}
&\int_{T}\|U\|_2^2\mathrm{d}x \leq \frac{3(1+\epsilon)}{2}\int_{T}\|U_\ell\|_2^2\mathrm{d}x+\left(1+\frac{1}{\epsilon}\right)\int_{T}\|R_1\|_2^2\mathrm{d}x,\\
&\int_{e}\|W\|_2^2\mathrm{d}s \leq \frac{3(1+\epsilon)}{2}\int_{e}\|W_\ell\|_2^2\mathrm{d}s+\left(1+\frac{1}{\epsilon}\right)\int_{e}\|R_2\|_2^2\mathrm{d}s.
\end{align*}
Thus \eqref{eqn:3.8} holds.

The invertibility of matrix \( B \) ensures that the matrix \( M \) is also invertible. By employing the aforementioned computations, we obtain the inequality 
\(
\|M^{-1}\|_2^2 \leq \|{(M^{-1})}^{\mathsf{H}} M^{-1}\|_2 = \|B^{-1}\|_2.
\)
Given that \( B \) is both positive and diagonalizable, we can apply equation \eqref{eqn:3.20} to conclude that 
\(
\|B^{-1}\|_2 = \frac{1}{\min_{1 \leq i \leq N} \sigma_i} \leq 2.
\)
Therefore $\|M^{-1}\|^2_2 \le 2$ and by Young's inequality with $\epsilon=10^{-2}$,  
\begin{align*}
&\int_{T}\|U_\ell\|_2^2\mathrm{d}x \leq 2(1+\epsilon)\int_{T}\|U\|_2^2\mathrm{d}x+2\left(1+\frac{1}{\epsilon}\right)\int_{T}\|R_1\|_2^2\mathrm{d}x,\\
&\int_{e}\|W_\ell\|_2^2\mathrm{d}s \leq 2(1+\epsilon)\int_{e}\|W\|_2^2\mathrm{d}s+2\left(1+\frac{1}{\epsilon}\right)\int_{e}\|R_2\|_2^2\mathrm{d}s.
\end{align*}
Thus \eqref{eqn:3.9} holds.

Similarly, we can obtain \eqref{eqn:3.8_adjoint} and \eqref{eqn:3.9_adjoint}.\\
\indent 
We will prove below that \(\sum_{\kappa\in\mathcal{T}_{\ell}} \vartheta_{\ell}(\kappa)^2\) and \(\sum_{\kappa\in\mathcal{T}_{\ell}} \vartheta^*_{\ell}(\kappa)^2\) are higher-order terms in comparison with \(\sum_{j\in J} \mu_{\ell}(u_j,\mathcal T_{\ell})^2\) and \(\sum_{j\in J} \mu^*_{\ell}(u^*_j,\mathcal T_{\ell})^2\).
From inverse estimate, (\ref{eqn:2.8}), and trace estimate, we deduce
\begin{align}
\nonumber
\sum_{T\in\mathcal{T}_{\ell}} \vartheta_{\ell}(T)^2
= & \sum_{T\in\mathcal{T}_{\ell}} \sum_{j\in {J}} \bigg( h^2_T\|u_j - \Lambda_{\ell} u_j\|^2_{0,T} + h^2_T\|\mathcal{L}({K}_{\ell}( u_j  - \Lambda_{\ell} u_j))\|^2_{0,T} \\
\nonumber
&\qquad + \sum_{e \in \mathcal{E}_{\ell}, e \subset \partial T} h_e\|[\![ \boldsymbol{A} \nabla (K_{\ell}(u_j - \Lambda_{\ell} u_j))\cdot \nu ]\!]\|^2_{0,e}\bigg)\\
\nonumber
\lesssim & \sum_{j\in {J}} (h_\ell^2 + 1 + A^2_{\max}) \|u_j  - \Lambda_{\ell} u_j\|^2_{0}\\
\label{eqn:3.21}
\lesssim & \sum_{j\in {J}} (h_\ell^2 + 1 + A^2_{\max})\left(\|u_j  - E_{\ell} u_j\|^2_{0} + \|u_j- G_{\ell} u_j\|^2_{0}\right).
\end{align}

From (7.3) in \cite{babuska1991}, the Nitsche technique, (\ref{eqn: regularity}), (\ref{eqn:3.1}), (\ref{eqn:3.3}), and (\ref{eqn:rel_a}),  we deduce
\begin{equation*}
\begin{aligned}
&\sum_{j\in {J}} \|u_j  - E_{\ell} u_j\|^2_{0} + \|u_j- G_{\ell} u_j\|^2_{0}  \lesssim
\sum_{j\in {J}} \|(K  - K_{\ell})|_{R(E)}\|^2_{0} +\sum_{j\in {J}} \|u_j- G_{\ell} u_j\|^2_{0}
\\
&\quad \lesssim N {h_\ell^{2r}}\|(K  - K_{\ell})|_{R(E)}\|^2_{V} + h_\ell^{2r}\sum_{j\in {J}} \|u_j- G_{\ell} u_j\|^2_{V}
\\
&\quad \lesssim N  h_\ell^{2r}\sum_{j\in {J}}\|(K - K_{\ell})u_{j}\|^2_{V} 
+ N h_\ell^{2r}\widehat{\delta}(R(E),R(E_{\ell}))^2
\\
&\quad \lesssim N {h_\ell^{2r}}\sum_{j\in J} \mu_{\ell}(u_j,\mathcal T_{{\ell}})^2.
\end{aligned}
\end{equation*} 
Thus collecting the above estimate into (\ref{eqn:3.21}) we have 
\begin{equation}\label{eqn:3.25b}
\begin{aligned}
\sum_{T\in\mathcal{T}_{\ell}} \vartheta_{\ell}(T)^2  &\lesssim  N h_\ell^{2r} \sum_{j\in J} \mu_{\ell}(u_j,\mathcal T_{{\ell}})^2.
\end{aligned}
\end{equation} 
Similarly we can obtain 
\begin{equation}\label{eqn:3.25c}
\begin{aligned}
\sum_{T\in\mathcal{T}_{\ell}} \vartheta^*_{\ell}(T)^2  &\lesssim  N h_\ell^{2r} \sum_{j\in J} \mu^*_{\ell}(u^*_j,\mathcal T_{\ell})^2.
\end{aligned}
\end{equation} 

Since $h_\ell$ is small enough, from (\ref{algorithm}), (\ref{eqn:3.8}), (\ref{eqn:3.8_adjoint}), (\ref{eqn:3.9}), and (\ref{eqn:3.9_adjoint}), we deduce that (\ref{eqn:3.10}) and (\ref{eqn:3.11}) hold true.  This completes the proof of the theorem.
\end{proof}

\section{Optimal convergence of the adaptive scheme}\label{sec4}

For any $m\in\mathbb{N}$, we denote by
\[
\mathbb T(m)=
  \{\mathcal{T}\in\mathbb T 
        : 
    \# \mathcal{T} - \#\mathcal{T}_0
               \leq m \}
\]
the set of admissible
triangulations in $\mathbb T$ whose cardinality differs from that
of $\mathcal{T}_0$ by $m$ or less. Let $\mathsf{W}=\{Ku_j: j\in J\}$ and $\mathsf{W}^*=\{K_*u^*_j: j\in J\}$ denote the generalized eigenfunction clusters.
Let $V_{\mathcal T}$ denote the finite-dimensional space defined on the mesh $\mathcal T$. The optimal convergence rate $s\in(0,+\infty)$ obtained by any
admissible mesh in $\mathbb{T}$ is characterized with respect to
\begin{equation*}
 |(\mathsf{W},\mathsf{W}^*)|_{\mathcal{A}_s} 
 = \sup_{m\in\mathbb{N}} m^s
      \inf_{\mathcal{T}\in\mathbb{T}(m)} (\Xi_{\mathcal T}+\Xi_{\mathcal T}^*),
\end{equation*}
where
\begin{align*}
\Xi_{\mathcal T} &= \sum_{j\in J}\inf_{\chi_j\in V_{\mathcal T}} \|Ku_j - \chi_j\|_V,\quad \Xi_{\mathcal T}^* = \sum_{j\in J}\inf_{\chi_j\in V_{\mathcal T}} \|K_*u^*_j - \chi_j\|_V.
\end{align*}
It is valid that $|(\mathsf{W},\mathsf{W}^*)|_{\mathcal{A}_s}<\infty$ if the rate of convergence $\Xi_{\mathcal T}+\Xi_{\mathcal T}^*=O(m^{-s})$ holds for the optimal triangulations $\mathcal T\in\mathbb{T}(m)$. The equivalence between $\Xi_{\mathcal T_\ell}+\Xi_{\mathcal T_\ell}^*$ and $\sum_{j\in J} (\| (K - K_\ell) u_j\|_V + \| (K_* - K_{\ell*}) u^*_j\|_V)$ is valid.\\

Lemma 5.4 of \cite{feischl2014} tell us that the following lemma holds. 
\begin{lemma}\label{lem:drel}
Assume that the mesh size of $\mathcal T_\ell$ is sufficiently small. 
Then there exists a constant ${C_{\mathrm{drel}}}>0$ such that
for all refinements $\mathcal T_\star\in \mathbb{T}$ of a triangulation $\mathcal T_\ell\in\mathbb{T} $, it holds
\begin{align}\label{eq:drel}
 \|K_\star u - K_\ell u\|_V \le{C_{\mathrm{drel}}}
\mu_\ell (u, \mathcal{T}_\ell \setminus \mathcal{T}_\star).
\end{align}
\end{lemma}
 
Proposition 3.6 of \cite{feischl2014} tell us that the following lemma holds. 
\begin{lemma}\label{l:qoI}
Assume the initial mesh size $h_0$ is small enough.
Then for all refinements $\mathcal T_\star\in \mathbb{T}$ of a triangulation $\mathcal T_\ell\in\mathbb{T}$, 
for any $\epsilon \in (0,1)$, there holds
\begin{equation}\label{eqn:lemma4.1}
\begin{aligned}
\|K_\star u- K_\ell u\|_V^2
&\leq 
\frac{1}{1-\epsilon}\|Ku - K_\ell u\|_V^2 - \| Ku- K_\star u\|_V^2.
\end{aligned}
\end{equation}
\end{lemma}

\begin{lemma}\label{l:reduction}
Assume the initial mesh size $h_0$ is small enough. Then there exist $\rho_1\in(0,1)$ and $\beta\in (0,+\infty)$ such that $\mathcal{T}_\ell$ and $\mathcal{T}_{\ell+1}$ generated by Algorithm 1 satisfy
 \begin{equation}\label{eqn:lemma4.2}
  \sum_{j\in J}\mu_{\ell+1}( u_j,\mathcal{T}_{\ell+1})^2
  \leq
  \rho_1 \sum_{j\in J}\mu_{\ell}(u_j,\mathcal{T}_\ell)^2
 +
 \beta \sum_{j\in J} \| (K_{\ell+1} - K_{\ell}) u_j\|_V^2 .
 \end{equation}
\end{lemma}

\begin{proof}
By the arguments as in Corollary 3.4 of \cite{cascon2008}, Lemma 3.1 of \cite{feischl2014},  and Lemma 6.10 of \cite{boffi2017}, we can deduce the lemma.
\end{proof}

\begin{lemma}\label{p:contraction}
 Assume the initial mesh size $h_0$ is chosen sufficiently small,
 there exist $\rho_2\in(0,1)$ and $\beta\in(0,+\infty)$ such that
 the term
\begin{equation}\label{e:xielldef}
\begin{aligned}
&\xi_\ell^2=  \sum_{j\in J}(\mu_\ell(u_j,\mathcal{T}_\ell)^2 + \mu^*_\ell(u^*_j,\mathcal{T}_\ell)^2)+ \beta \sum_{j\in J} (\| (K - K_\ell) u_j\|_V^2 + \| (K_* - K_{\ell*}) u^*_j\|_V^2
\end{aligned}
\end{equation}
 satisfies
 \begin{equation}\label{e:xielldef_a}
   \xi_{\ell+1}^2 \leq \rho_2 \xi_{\ell} ^2
        \quad \forall \ell\in\mathbb{N}.
 \end{equation}
\end{lemma}
\begin{proof}
The proof uses the arguments as in Proposition 6.11 of \cite{boffi2017}. 
Let
\begin{equation*}
\begin{aligned}
&\mu_{\ell}^2
        =\sum_{j\in J} (\mu_\ell(u_j,\mathcal{T}_\ell)^2 + \mu^*_\ell(u^*_j,\mathcal{T}_\ell)^2),\\
&e^2_\ell = \sum_{j\in J} (\| (K - K_\ell) u_j\|_V^2 + \| (K_* - K_{\ell*}) u^*_j\|_V^2).
\end{aligned}
\end{equation*}
From (\ref{eqn:lemma4.1}) and (\ref{eqn:lemma4.2}), we have
\begin{equation*}
 \mu_{\ell+1}^2
  + \beta e_{\ell+1}^2
  \leq
  \rho_1 \mu_{\ell}^2
  +
  \frac{\beta}{1-\epsilon} e_\ell^2,
\end{equation*}
from which, (\ref{eqn:rel_a}), and (\ref{eqn:rel_b}), we have
 \begin{equation*}
\mu_{\ell+1}^2
  + \beta e_{\ell+1}^2
  \leq
  (\rho_1 + \epsilon^{\prime} {C^2_{\mathrm{rel}}} \beta) \mu_{\ell}^2
  +
 \beta(\frac{1}{1-\epsilon}-\epsilon^{\prime})e_\ell^2\quad\forall \epsilon^{\prime}\in (0,1).
 \end{equation*}
 We take
\begin{equation*}
\rho_2 = \max \{
 \rho_1 + \epsilon^{\prime} {C^2_{\mathrm{rel}}} \beta,
 \frac{1}{1-\epsilon}-\epsilon^{\prime} 
\},
\end{equation*}
so that
\begin{equation*}
\mu_{\ell+1}^2 + \beta e_{\ell+1}^2
\leq 
\rho_2 (\mu_{\ell}^2+ \beta e_\ell^2).
\end{equation*}
Choosing $\epsilon^{\prime}$ small enough and $\epsilon$ small enough, we obtain $\rho_2 < 1$.
\end{proof} 

\begin{theorem}\label{thm:4.5}
Assume the initial mesh size $h_0$ and the bulk parameter $\theta$ are small
enough, and $|(\mathsf{W},\mathsf{W}^*)|_{\mathcal{A}_s}<\infty$. Then there holds
\begin{equation}\label{eqn:4.2}
\sum_{j\in J} (\| (K - K_\ell) u_j\|_V + \| (K_* - K_{\ell*}) u^*_j\|_V)
\lesssim (\# \mathcal{T}_\ell - \# \mathcal{T}_0 )^{-s} |(\mathsf{W},\mathsf{W}^*)|_{\mathcal{A}_s}.
\end{equation}
\end{theorem}
\begin{proof}
We follow the lines of the proofs of Theorem 3.1 of \cite{gallistl2015} and Theorem 4.1 of \cite{boffi2017}.\\
\indent For $0<\tau\leq |(\mathsf{W},\mathsf{W}^*)|_{\mathcal{A}_s}^2/\xi_0^2$, we define $\varepsilon(\ell)=\sqrt{\tau}\,\xi_\ell$ (see (\ref{e:xielldef}) for the definition of $\xi_\ell$ and let $\xi_0\neq0$). 
Let $N(\ell)\in\mathbb{N}$ be minimal with the property
\begin{equation*}
  |(\mathsf{W},\mathsf{W}^*)|_{\mathcal{A}_s}^2
 \leq \varepsilon(\ell)^2\,N(\ell)^{2s}.
\end{equation*}
It is clear that $N(\ell)>1$, otherwise
\begin{equation*}
|(\mathsf{W},\mathsf{W}^*)|_{\mathcal{A}_s}\le\varepsilon(\ell),
\end{equation*}
but this and the definition of $\varepsilon(\ell)$ contradict with (\ref{e:xielldef}).\\
\indent From the minimality of $N(\ell)$ it turns out that
\begin{equation}\label{e:optNellBound}
N(\ell)
\leq 2 |(\mathsf{W},\mathsf{W}^*)|_{\mathcal{A}_s}^{1/s}
            \varepsilon(\ell)^{-1/s}
\quad\text{for all }\ell\in\mathbb{N}_0.
\end{equation}
Let $\widetilde{\mathcal{T}}_\ell\in\mathbb{T}$ denote the optimal triangulation of cardinality
\begin{equation*}
 \# \widetilde{\mathcal{T}}_\ell
 \leq \# \mathcal{T}_0 + N(\ell)
\end{equation*}
in the sense that the discrete solution operators
$\widetilde K_\ell$ and $\widetilde K_{\ell*}$ defined on the mesh $\widetilde{\mathcal{T}}_\ell$
satisfies
\begin{equation} \label{e:tildeToptimal}
  \sum_{j\in J} \| (K - \widetilde K_\ell) u_j\|_V^2 
+ \| (K_* - \widetilde K_{\ell*}) u^*_j\|_V^2 
\leq 
    N(\ell)^{-2s} |(\mathsf{W},\mathsf{W}^*)|_{\mathcal{A}_s}^2
\leq \varepsilon(\ell)^2.
\end{equation}
Let us consider the overlay $\widehat{\mathcal{T}}_\ell$, which is the smallest common
refinement of $\mathcal{T}_\ell$ and $\widetilde{\mathcal{T}}_\ell$. 
Thanks to Lemma 3.7 of \cite{cascon2008}, we have
\begin{equation}\label{e:overlayProperty}
\#(\mathcal{T}_\ell \setminus\widehat{\mathcal{T}}_\ell)\le
  \# \widehat{\mathcal{T}}_\ell  
 - \# \mathcal{T}_\ell
\leq 
   \# \widetilde{\mathcal{T}}_\ell -
        \# \mathcal{T}_0
\leq 
   N(\ell),
\end{equation}
which together with \eqref{e:optNellBound}
yielding
\begin{equation}\label{e:OptEst1}
  \# (\mathcal{T}_\ell \setminus\widehat{\mathcal{T}}_\ell)
  \leq N(\ell)
 \leq 2 |(\mathsf{W},\mathsf{W}^*)|_{\mathcal{A}_s}^{1/s}
            \varepsilon(\ell)^{-1/s}.
\end{equation}
Let $\widehat K_\ell$ and $\widehat K_{\ell*}$ denote the discrete solution operators defined on the mesh $\widehat{\mathcal{T}}_\ell$, respectively.

From (\ref{eqn:lemma4.1})
with $\mathcal{T}_\ell=\widetilde{\mathcal{T}}_\ell$ and $\mathcal{T}_\star=\widehat{\mathcal{T}}_\ell$, we have
\begin{equation*}
\begin{aligned}
&
\sum_{j\in J}  (\| (K - \widehat K_\ell) u_j \|_V^2 + \| (K_* - \widehat K_{\ell*}) u^*_j \|_V^2)\\
&\qquad \leq
\frac{1}{1-\epsilon}
\sum_{j\in J}  (\| (K - \widetilde K_\ell) u_j \|_V^2 + \| (K_* - \widetilde K_{\ell*}) u^*_j \|_V^2),
\end{aligned}
\end{equation*}
from which, $1/(1-\epsilon) \leq 2$, and \eqref{e:tildeToptimal}, we obtain
\begin{equation}\label{e:step2}
 \sum_{j\in J} (\| (K - \widehat K_\ell) u_j \|_V^2 + \| (K_* - \widehat K_{\ell*}) u^*_j \|_V^2)
 \leq  2\varepsilon(\ell)^2.
\end{equation}

Now, we show there exists a constant $C_1$ such that
\begin{equation}
\sum_{j\in J}
 (\mu_\ell(u_j,\mathcal{T}_\ell)^2 + \mu^*_\ell(u^*_j,\mathcal{T}_\ell)^2)
   \leq
    C_1
    \sum_{j\in J}
    (\mu_\ell(u_j,\mathcal{T}_\ell\setminus\widehat{\mathcal{T}}_\ell)^2 + \mu^*_\ell(u^*_j,\mathcal{T}_\ell\setminus\widehat{\mathcal{T}}_\ell)^2).
\label{eq:keyargument}
\end{equation}
From the triangle inequality and (\ref{eq:drel}), we obtain for any $j\in J$,
\begin{equation*}
\begin{split}
\| (K - K_\ell) u_j \|_V^2
&\leq
2 \| (K - \widehat K_\ell) u_j \|_V^2
 + 2 \| (\widehat K_\ell - K_\ell) u_j \|_V^2
\\
&\leq
2 \| (K - \widehat K_\ell) u_j\|_V^2
+
2 {C^2_{\mathrm{drel}}} \mu_\ell(u_j, \mathcal{T}_\ell\setminus\widehat{\mathcal{T}}_\ell)^2,
\end{split}
\end{equation*}
and 
\begin{equation*}
\begin{split}
\| (K_* - K_{\ell*}) u^*_j \|_V^2
&\leq
2 \| (K_* - \widehat K_{\ell*}) u^*_j \|_V^2
 + 2 \| (\widehat K_{\ell*} - K_{\ell*}) u^*_j \|_V^2
\\
&\leq
2 \| (K_* - \widehat K_{\ell*}) u^*_j\|_V^2
+
2 {C^2_{\mathrm{drel}}} \mu^*_\ell(u^*_j, \mathcal{T}_\ell\setminus\widehat{\mathcal{T}}_\ell)^2.
\end{split}
\end{equation*}
Assuming that $h_0$ is sufficiently small, this leads to some constant $C_2$ such that with
\eqref{e:step2} it follows
\begin{equation*}
\begin{aligned}
&\sum_{j\in J}
(\| (K - K_\ell) u_j \|_V^2 + \| (K_* - K_{\ell*}) u^*_j \|_V^2)\\
&\qquad \leq
C_2 \varepsilon(\ell)^2
+
C_2 {C^2_{\mathrm{drel}}} \sum_{j\in J}(\mu_\ell(u_j,\mathcal{T}_\ell \setminus \widehat{\mathcal{T}}_\ell)^2 + \mu^*_\ell(u^*_j,\mathcal{T}_\ell\setminus\widehat{\mathcal{T}}_\ell)^2).
\end{aligned}
\end{equation*}

Let $C_{\mathrm{eq}}$ denote the constant of
$C_2\xi_\ell^2 \leq C_{\mathrm{eq}} \sum_{j\in J} (\mu_\ell(u_j,\mathcal{T}_\ell)^2 + \mu^*_\ell( u^*_j,\mathcal{T}_\ell)^2)$.
From (\ref{eqn:eff_a}),  (\ref{eqn:eff_b}),  the definition of $\varepsilon(\ell)$, and the above estimates, we deduce
\begin{equation*}
 \begin{aligned}
 &C_{\mathrm{eff}}^{-2}
   \sum_{j\in J} (\mu^2_\ell(u_j,\mathcal{T}_\ell) + \mu^*_\ell(u^*_j,\mathcal{T}_\ell)^2)\\
 &\leq
 C_2 \varepsilon(\ell)^2
+
C_2 {C^2_{\mathrm{drel}}} \sum_{j\in J}(\mu^2_\ell(u_j,\mathcal{T}_\ell\setminus\widehat{\mathcal{T}}_\ell)
+ \mu^*_\ell(u^*_j,\mathcal{T}_\ell\setminus\widehat{\mathcal{T}}_\ell)^2)
 \\
 &\leq
    \tau C_{\mathrm{eq}}
    \sum_{j\in J}(\mu^2_\ell(u_j,\mathcal{T}_\ell)+\mu^*_\ell(u^*_j,\mathcal{T}_\ell)^2)
 +
C_2 {C^2_{\mathrm{drel}}} \sum_{j\in J}(\mu^2_\ell(u_j,\mathcal{T}_\ell\setminus\widehat{\mathcal{T}}_\ell)
+ \mu^*_\ell(u^*_j,\mathcal{T}_\ell\setminus\widehat{\mathcal{T}}_\ell)^2).
 \end{aligned}
\end{equation*}
Letting
$C_1 = 
  (C_{\mathrm{eff}}^{-2}-\tau C_{\mathrm{eq}})^{-1}
  C_2{C^2_{\mathrm{drel}}} >0
$, we
obtain~\eqref{eq:keyargument}.

Thanks to (\ref{eqn:3.11}), we choose
\begin{equation*}
0<\theta\leq \frac{1}{4C_1}.
\end{equation*}

Select $\mathcal{M}_\ell \subseteq \mathcal{T}_\ell$ in step (iii) of Algorithm 1  with minimal
cardinality such that 
\begin{equation*}
\sum_{j\in J}(\eta_{\ell}(u_{j,\ell}, \mathcal M_\ell)^2 + \eta^*_{\ell}(u^*_{j,\ell}, \mathcal M_\ell)^2)
\geq
\theta\sum_{j\in J}(\eta_{\ell}(u_{j,\ell}, \mathcal{T}_\ell)^2 + \eta^*_{\ell}(u^*_{j,\ell}, \mathcal{T}_\ell)^2).
\end{equation*}
Estimate~\eqref{eq:keyargument} and the definition of $\theta$
imply together with (\ref{eqn:3.10})
that also $\mathcal{T}_\ell\setminus\widehat{\mathcal{T}}_\ell$
satisfies the bulk criterion, that is
\begin{equation*}
\sum_{j\in J}(\eta_{\ell}(u_{j,\ell}, \mathcal{T}_\ell\setminus\widehat{\mathcal{T}}_\ell)^2 + \eta^*_{\ell}(u^*_{j,\ell}, \mathcal{T}_\ell\setminus\widehat{\mathcal{T}}_\ell)^2)
\geq
\theta\sum_{j\in J} (\eta_{\ell}(u_{j,\ell}, \mathcal{T}_\ell)^2 + \eta^*_{\ell}(u^*_{j,\ell}, \mathcal{T}_\ell)^2).
\end{equation*}
The minimality of $\mathcal{M}_\ell$ and
\eqref{e:OptEst1} show that
\begin{equation}\label{e:minimalMl}
  \# \mathcal M_\ell
  \leq
  \# (\mathcal{T}_\ell\setminus\widehat{\mathcal{T}}_\ell)
  \leq
  2 |(\mathsf{W},\mathsf{W}^*)|_{\mathcal A_s}^{1/s}
         \tau^{-1/(2s)} \xi_\ell^{-1/s}.
\end{equation}
Thanks to \cite{binev2004,stevenson2008}, there exists a constant $C_{\mathrm{BDV}}$ such that
\begin{equation}\label{thm4.4_formula_a}
 \begin{aligned}
  \# \mathcal{T}_\ell
   - \# \mathcal{T}_0
   &\leq
   C_{\mathrm{BDV}}
    \sum_{k=0}^{\ell-1} \# \mathcal M_k
\leq
   2C_{\mathrm{BDV}}   
   |(\mathsf{W},\mathsf{W}^*)|_{\mathcal{A}_s}^{1/s}  
         \tau^{-1/(2s)} 
    \sum_{k=0}^{\ell-1}\xi_k^{-1/s}.
 \end{aligned}
\end{equation}
From (\ref{e:xielldef_a}), we have
$\xi_\ell^2 \leq \rho_2^{\ell-k} \xi_k^2$ for 
$k=0,\ldots,\ell$.
Since $\rho_2<1$, a geometric series argument leads to
\begin{equation}\label{thm4.4_formula_b}
 \sum_{k=0}^{\ell-1} \xi_k ^{-1/s}
 \leq
 \xi_\ell ^{-1/s} 
   \sum_{k=0}^{\ell-1} \rho_2^{(\ell-k)/(2s)}
 \leq
 \xi_\ell ^{-1/s} 
       \rho_2^{1/(2s)} \left/ \left(1-\rho_2^{1/(2s)}\right)\right. .
\end{equation}
From (\ref{thm4.4_formula_a}) and (\ref{thm4.4_formula_b}), we arrive at
\begin{equation*}
 \begin{aligned}
   &
     \# \mathcal{T}_\ell
    - \# \mathcal{T}_0 
   \leq
      2C_{\mathrm{BDV}}   
      |(\mathsf{W},\mathsf{W}^*)|_{\mathcal A_s}^{1/s} 
         \tau^{-1/(2s)}  \xi_\ell ^{-1/s} 
       \rho_2^{1/(2s)} \left/ \left(1-\rho_2^{1/(2s)}\right)\right..
 \end{aligned}
\end{equation*}
From (\ref{eqn:rel_a}), (\ref{eqn:rel_b}), (\ref{eqn:eff_a}), and (\ref{eqn:eff_b}), we obtain the equivalence between $\xi_\ell^2$ and the error 
 $\sum_{j\in J} (\| (K - K_\ell) u_j\|_V^2 + \| (K_* - K_{\ell*}) u^*_j\|_V^2)$. 
 Thus the assertion is obtained.
\end{proof} 

\begin{theorem}\label{c:opt}
Under the conditions of Theorem \ref{thm:4.5}, assuming that the condition (C1) or (C2) is valid, there hold that
\begin{align}
\label{eqn:3.28}
&\widehat{\delta}(R(E),R(E_{\ell}))\lesssim (\# \mathcal T_{\ell} - \# \mathcal T_{0})^{-s}|(\mathsf{W},\mathsf{W}^*)|_{\mathcal{A}_s},\\
\label{eqn:3.29}
&\widehat{\delta}(R(E_*),R(E_{\ell*}))\lesssim (\# \mathcal T_{\ell} - \# \mathcal T_{0})^{-s}|(\mathsf{W},\mathsf{W}^*)|_{\mathcal{A}_s},\\
\label{eqn:3.30}
&|\widehat{\lambda} - \widehat{\lambda}_{\ell}|  \lesssim N^2(\# \mathcal T_{\ell} - \# \mathcal T_{0})^{-2s} |(\mathsf{W},\mathsf{W}^*)|^2_{\mathcal{A}_s}.
\end{align}
When the eigenvalue cluster $\{\lambda_{j}\}_{j\in J}$ is composed of non-defective eigenvalues, then
\begin{align}
\label{eqn:3.31}
&|\lambda_k - \lambda_{k,\ell}|  \lesssim N^2(\# \mathcal T_{\ell} - \# \mathcal T_{0})^{-2s}
|(\mathsf{W},\mathsf{W}^*)|^2_{\mathcal{A}_s} \quad \text{ for } k\in J.
\end{align}
\end{theorem} 

\begin{proof}
Combining (\ref{eqn:3.3}) with (\ref{eqn:4.2}), we obtain (\ref{eqn:3.28}).
Combining (\ref{eqn:3.4}) with (\ref{eqn:4.2}), we obtain (\ref{eqn:3.29}).
Combining (\ref{eqn:3.5}) with (\ref{eqn:4.2}), we obtain (\ref{eqn:3.30}). 
When the eigenvalue cluster $\{\lambda_{j}\}_{j\in J}$ is composed of non-defective eigenvalues, thanks to (\ref{eqn:3.5a}) and (\ref{eqn:4.2}), we can obtain (\ref{eqn:3.31}).
\end{proof}

\section{Numerical experiments}\label{sec5}
In this section, we will present a numerical example to support the theoretical analysis.
Our program is completed using the Python package of scikit-fem \cite{skfem2020}. \\
\indent We consider the Kellogg problem, with the domain $\Omega=\bigcup_{i=1}^4\Omega_i$ where $\Omega_{1}=(0,1)^{2}$, $\Omega_{2}=(-1, 0)\times (0,1)$, $\Omega_{3}=(-1, 0)^{2}$, $\Omega_{4}=(0,1)\times (-1,0)$. The operator $\mathcal{L} v=-\nabla\cdot(\boldsymbol{A} v)+\boldsymbol{b} \cdot \nabla v$ is defined with  $\boldsymbol{b}=(2,2)$, where $\boldsymbol{A}=10\boldsymbol{I}$ in $\Omega_1$ and $\Omega_3$, and $\boldsymbol{A}=\boldsymbol{I}$ in $\Omega_2$ and $\Omega_4$.\\
\indent From \cite{giani2016} we know the first eigenvalue $17.7143168$ with an accuracy of $10^{-6}$.
The first twelve ``exact" eigenvalues
\begin{align*}
&\lambda_1 = 17.714 316 836 537, \lambda_2 = 20.741 585 348 761, \lambda_3 = 37.145 042 894 655,\\ 
&\lambda_4 = 43.608 009 384 122,
\lambda_5 = 48.640 297 883 881, 
\lambda_6 = 49.129 389 042 157, \\
&\lambda_7 = 63.720 910 445 531, 
\lambda_8 = 69.110 565 445 000,
\lambda_9 = 77.939 634 255 303,\\ 
&\lambda_{10} = 78.541 679 776 972, 
\lambda_{11} = 94.585 833 879 139,
\lambda_{12} = 94.921 224 922 705
\end{align*}
are obtained by using the $P_3$ element on an adaptively refined mesh by Algorithm~1 with $dof>10^{5}$ where the initial mesh is shown in Fig.~\ref{fig1b} (left) and an adaptively refined mesh is shown in Fig.~\ref{fig1b} (right).\\
\indent For the bulk parameter $\theta\in\{0.25,0.5,0.75,1.0\}$, the curves of global error estimators $\sum_{j\in J}(\eta_{\ell}(u_{j,\ell}, \mathcal T_\ell)^2 + \eta_{\ell}^*(u^*_{j,\ell}, \mathcal T_\ell)^2)$ are depicted on Fig.~\ref{fig2b} for the clustered eigenvalues $\{\lambda_{j,\ell}\}_{j\in J}$ ($n=0$ and $N=12$) by using the $P_p$ ($p=1,2,3$) element.
It is shown that $\sum_{j\in J}(\eta_{\ell}(u_{j,\ell}, \mathcal T_\ell)^2 + \eta_{\ell}^*(u^*_{j,\ell}, \mathcal T_\ell)^2)$ is cluster-robust and reliable.

We depict the error curves of the approximate eigenvalue cluster $\{\lambda_{j,\ell}\}_{j\in J}$ by using the $P_p$ ($p=1,2,3$) element in Figs.~\ref{fig3b}--\ref{fig5b}, from which it is shown that the first twelve approximate eigenvalues are non-defective, as supported by the numerical evidence in \cite{gasser2019}, and that the error curves are approximately parallel to the line with slope $-p$ ($p=1,2,3$).
The errors of the cluster reach optimal convergence rates. 
The numerical results are consistent with the theoretical results.

Different eigenvalues in the cluster $\{\lambda_{j,\ell}\}_{j\in J}$ exhibit different numerical sensitivity and digits. Although the trends of the errors of the first two approximate eigenvalues are unstable with respect to the values of $\theta$, the trends of the errors of the rest eigenvalues are stable in Figs.~\ref{fig3b}--\ref{fig5b}.

\begin{figure*}
\centering
\includegraphics[width=1.5in]{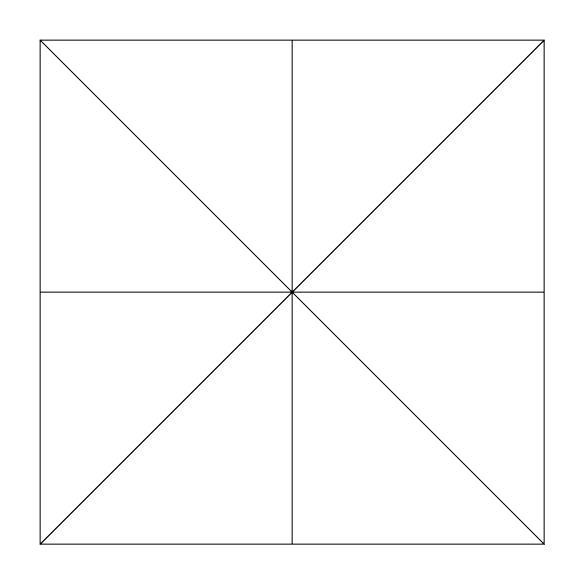}
\includegraphics[width=1.5in]{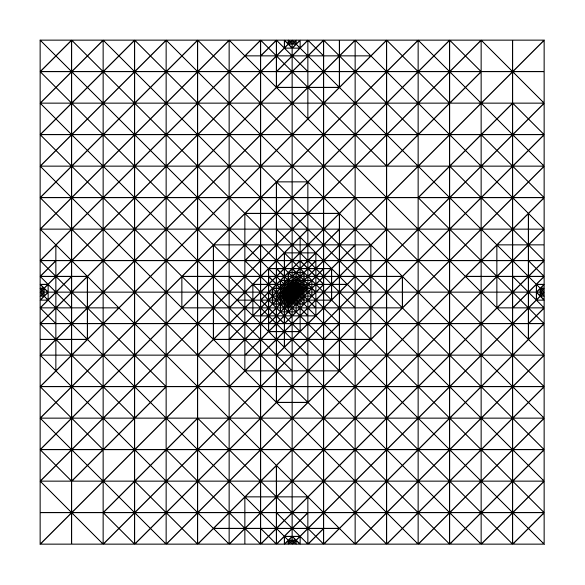}
\label{fig1b}
\caption{The initial mesh (left) and adaptively refined mesh by using the $P_{3}$ element (right) with the bulk parameter $\theta=0.5$.}
\end{figure*}

\begin{figure*}
\centering
\includegraphics[width=2.0in]{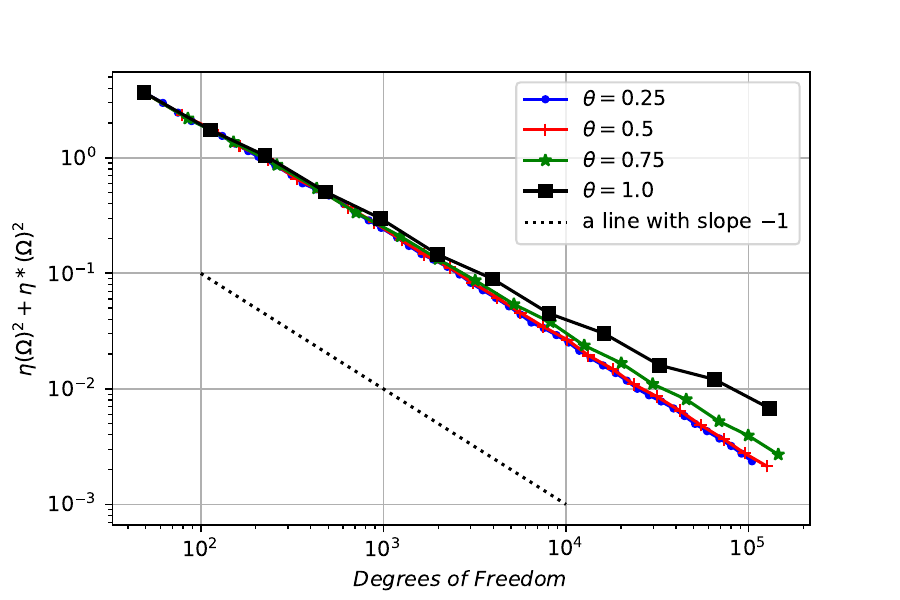}
\includegraphics[width=2.0in]{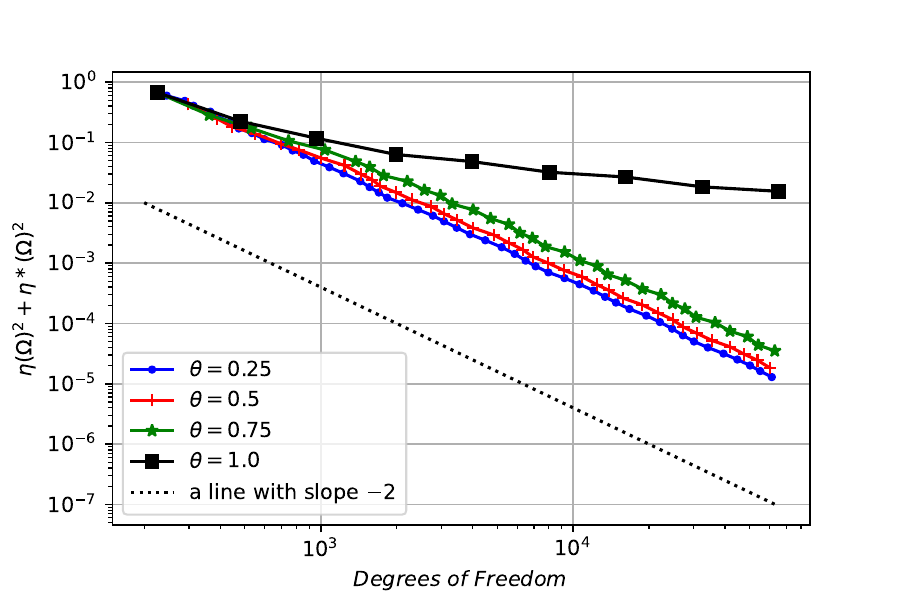}
\includegraphics[width=2.0in]{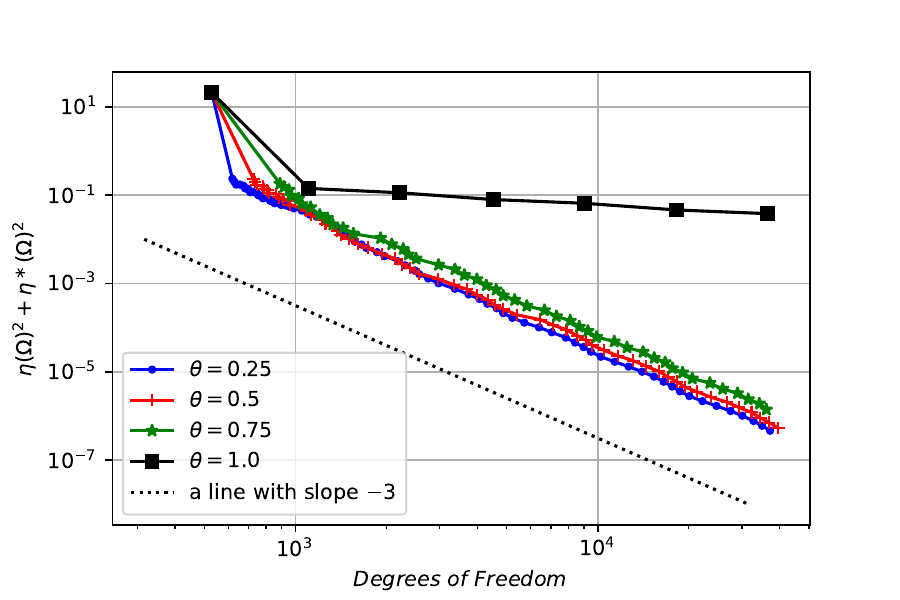}
\caption{\label{fig2b} The global error estimators $\sum_{j\in J}(\eta_{\ell}(u_{j,\ell}, \mathcal T_\ell)^2 + \eta_{\ell}^*(u^*_{j,\ell}, \mathcal T_\ell)^2)$ for the clustered eigenvalues $\{\lambda_{j,\ell}\}_{j\in J}$ ($n=0$ and $N=12$) by using the $P_1$ element (top left), the $P_2$ element (top right), and the $P_3$ element (bottom), respectively.}
\end{figure*}

\begin{figure*}
\centering
\includegraphics[width=2.0in]{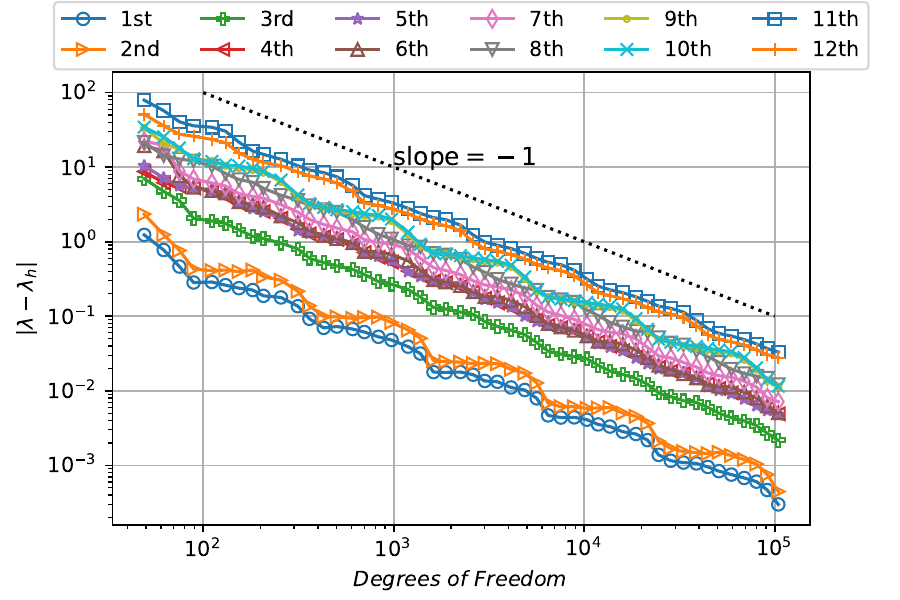}
\includegraphics[width=2.0in]{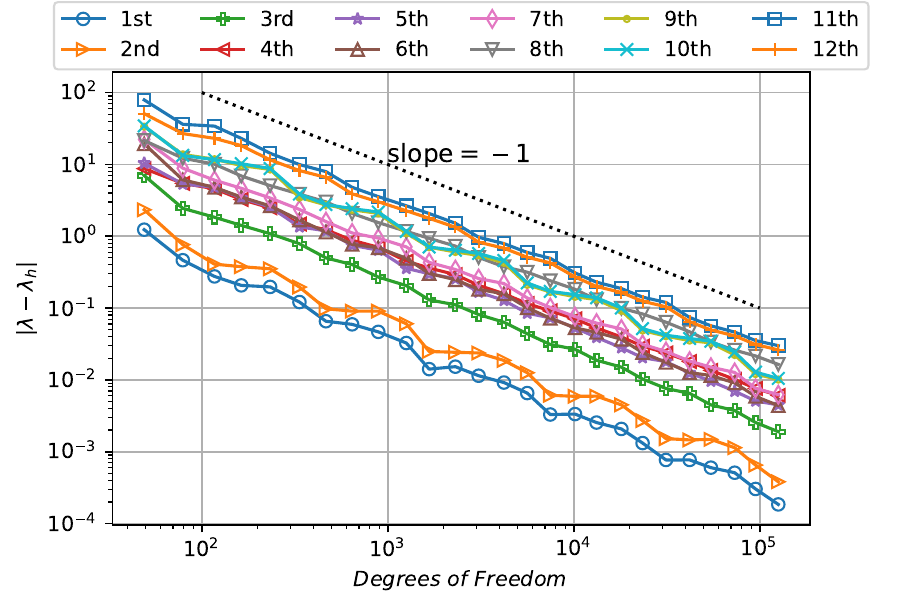}
\includegraphics[width=2.0in]{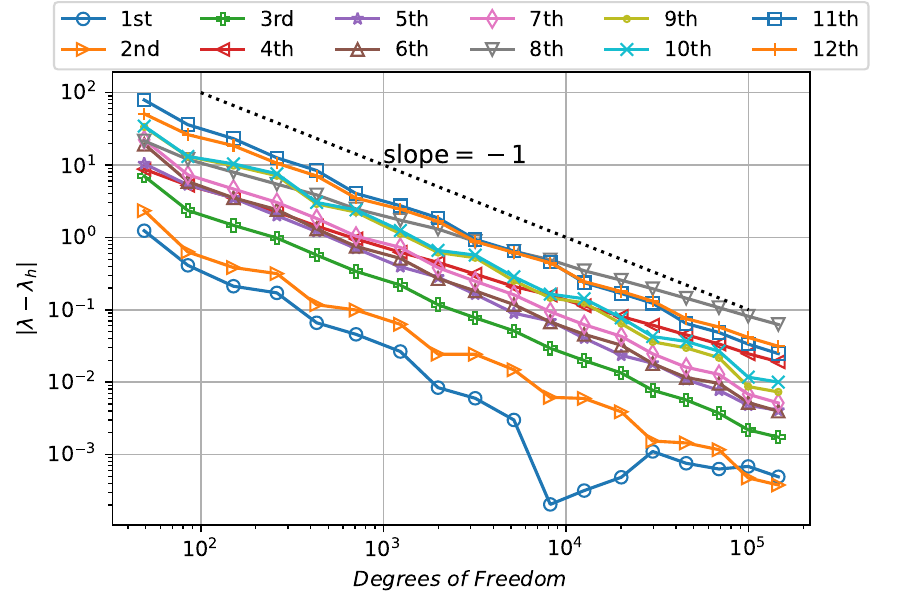}
\includegraphics[width=2.0in]{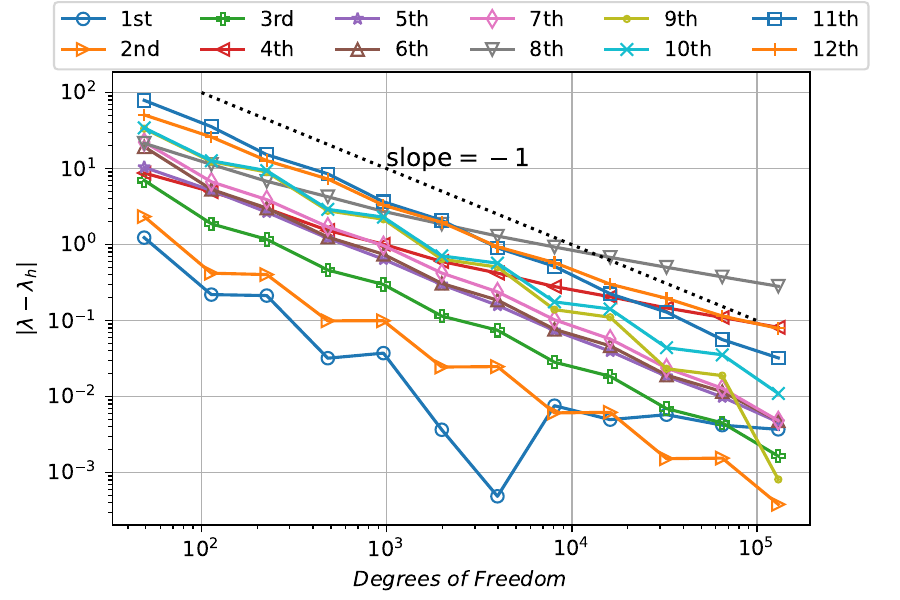}
\caption{\label{fig3b}The error curves of the clustered eigenvalues $\{\lambda_{j,\ell}\}_{j\in J}$ ($n=0$ and $N=12$) on adaptively refined meshes obtained by the $P_1$ element when the bulk parameters $\theta=0.25$ (top left),  $\theta=0.5$ (top right),  $\theta=0.75$ (bottom left), and $\theta=1$ (bottom right), respectively.}
\end{figure*}

\begin{figure*}
\centering
\includegraphics[width=2.0in]{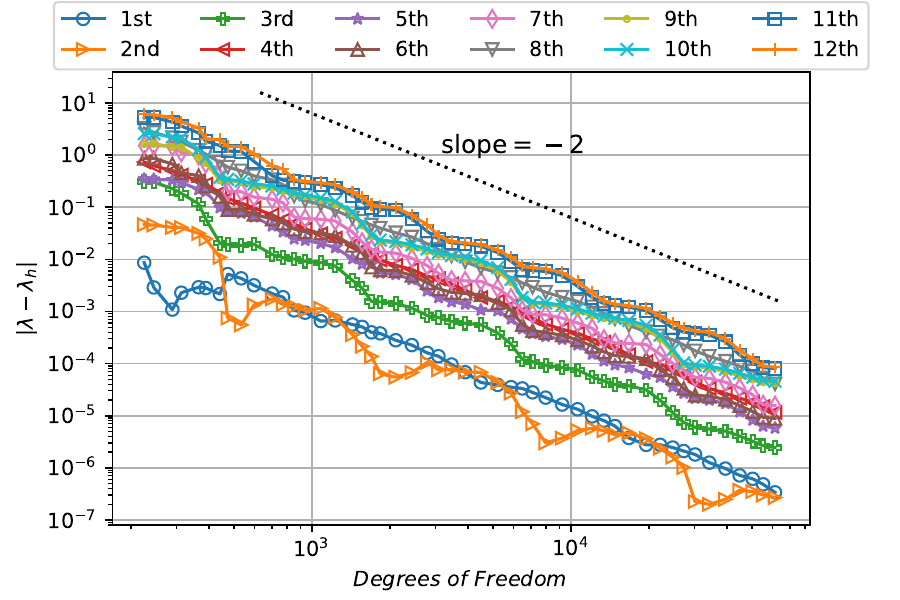}
\includegraphics[width=2.0in]{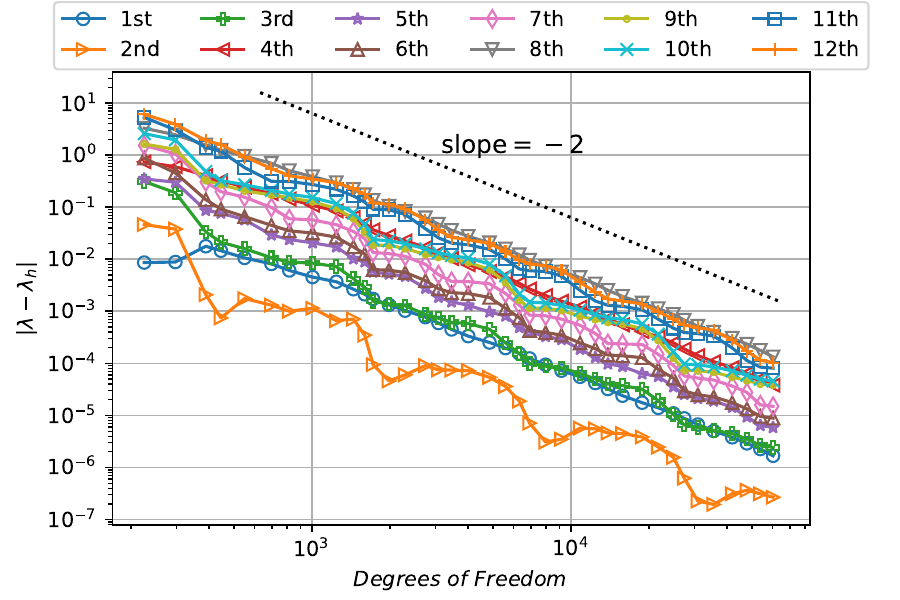}
\includegraphics[width=2.0in]{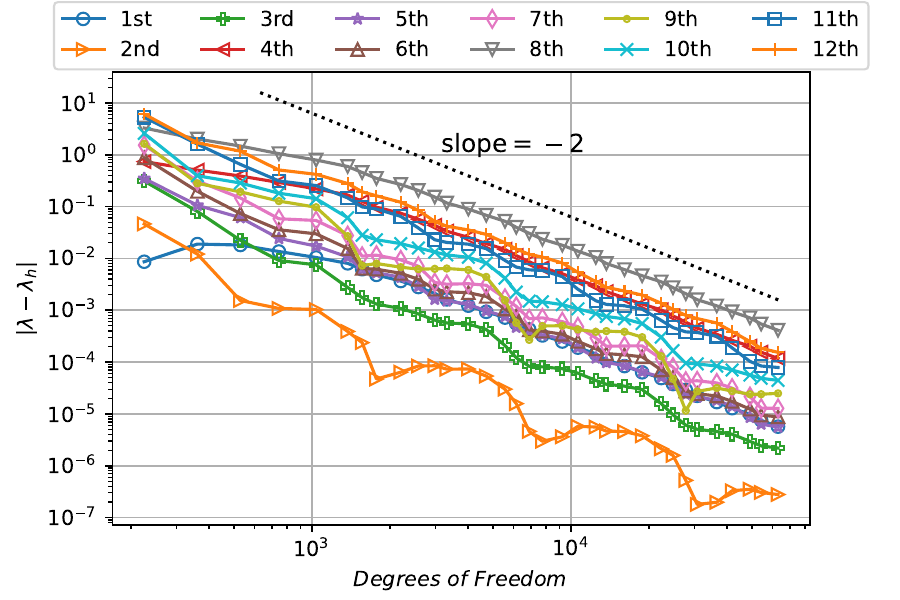}
\includegraphics[width=2.0in]{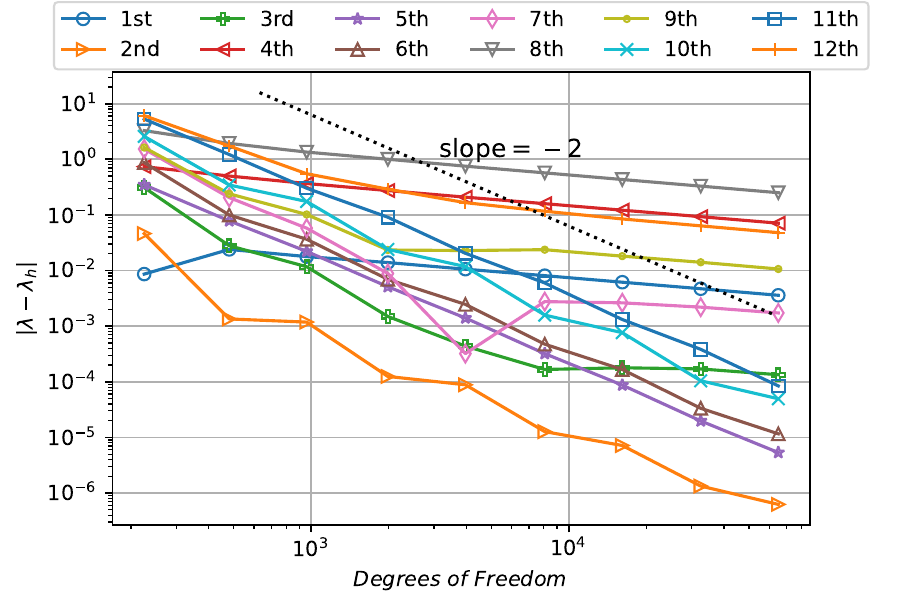}
\caption{\label{fig4b}The error curves of the clustered eigenvalues $\{\lambda_{j,\ell}\}_{j\in J}$ ($n=0$ and $N=12$) on adaptively refined meshes obtained by the $P_2$ element when the bulk parameters $\theta=0.25$ (top left),  $\theta=0.5$ (top right),  $\theta=0.75$ (bottom left), and $\theta=1$ (bottom right), respectively.}
\end{figure*}

\begin{figure*}
\centering
\includegraphics[width=2.0in]{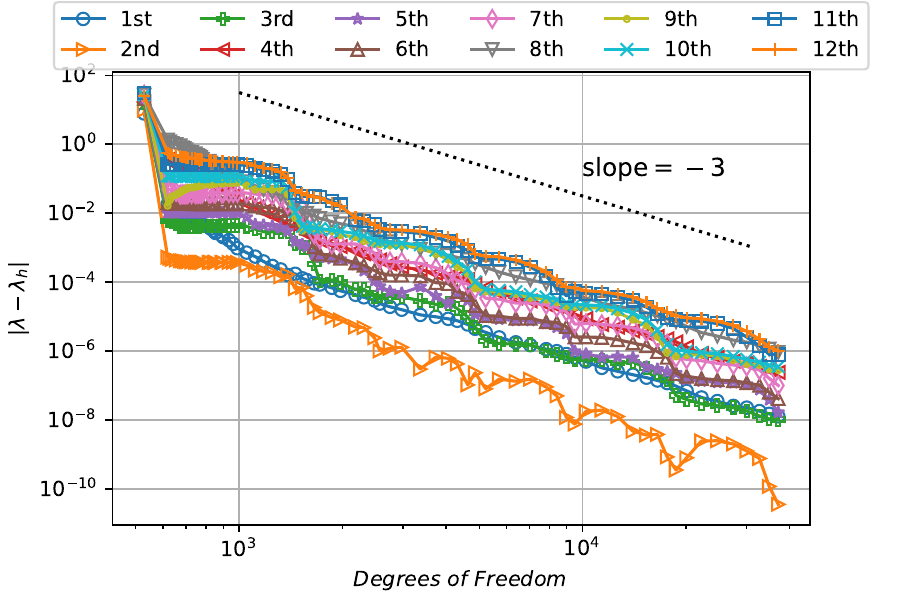}
\includegraphics[width=2.0in]{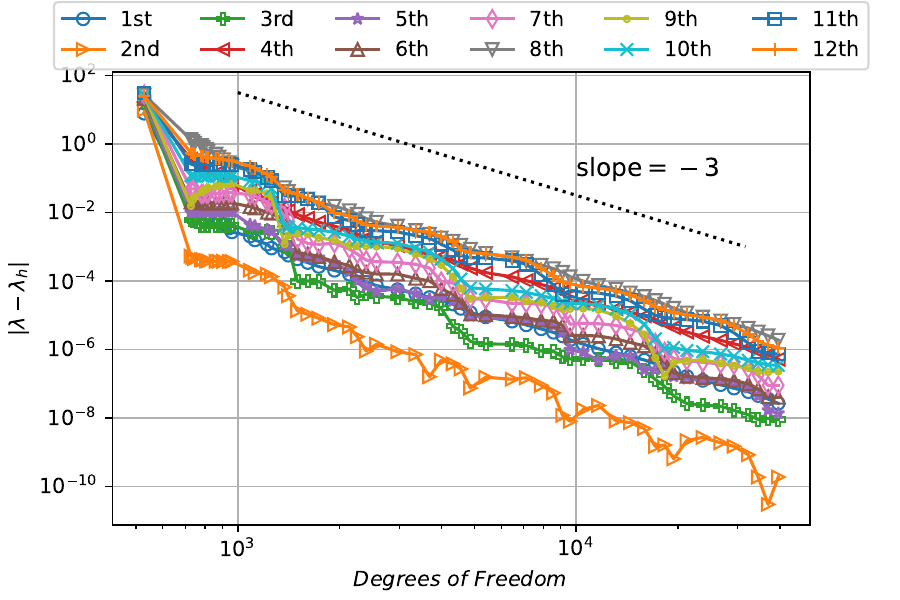}
\includegraphics[width=2.0in]{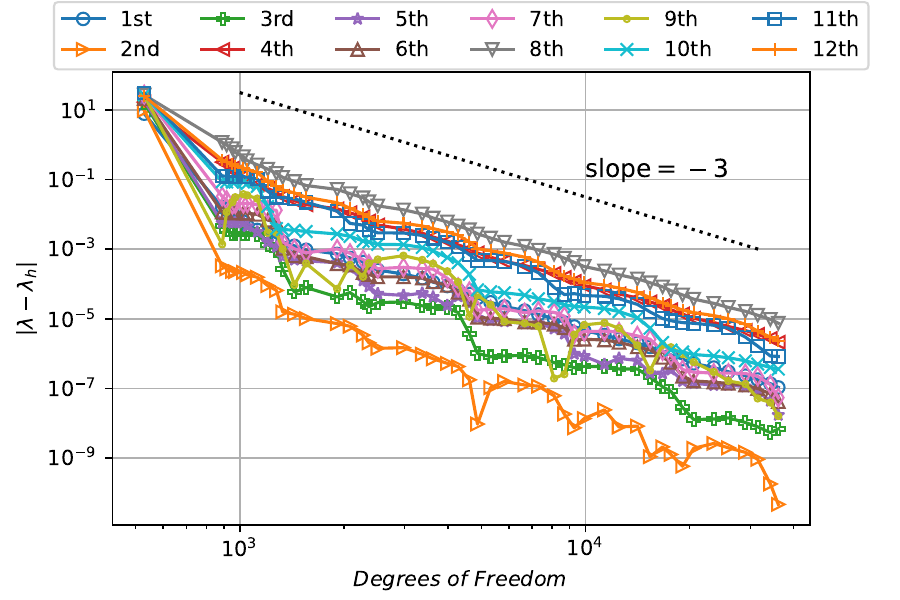}
\includegraphics[width=2.0in]{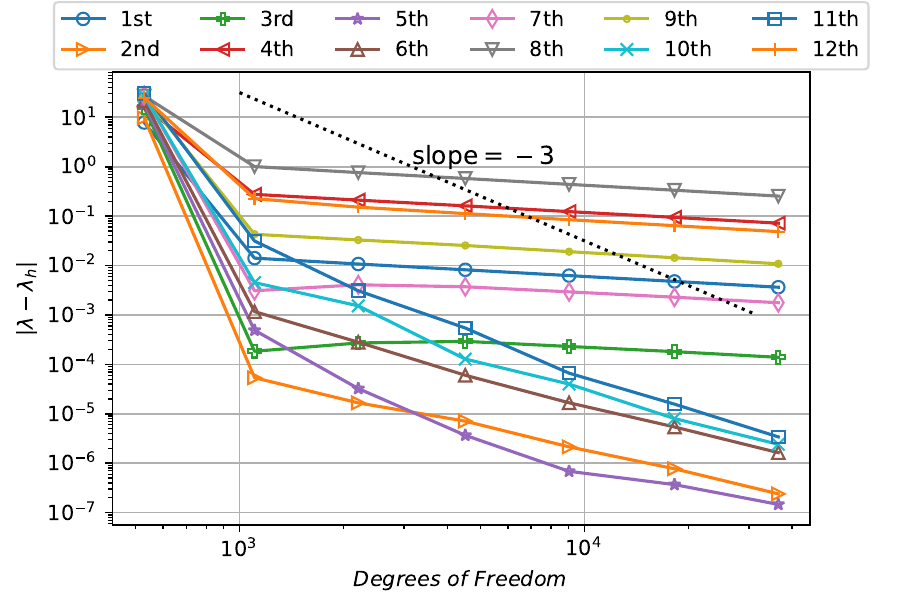}
\caption{\label{fig5b} The error curves of the clustered eigenvalues  $\{\lambda_{j,\ell}\}_{j\in J}$ ($n=0$ and $N=12$)  on adaptively refined meshes obtained by the $P_3$ element when the bulk parameters $\theta=0.25$ (top left),  $\theta=0.5$ (top right),  $\theta=0.75$ (bottom left), and $\theta=1$ (bottom right), respectively.}
\end{figure*}

\section{Conclusion comments}\label{sec:6}
In this paper, we first prove the optimal convergence of adaptive finite element methods of non-self-adjoint eigenvalue problems for the eigenvalue cluster.
The proof technique differs from those used for the optimal convergence of adaptive finite element methods for self-adjoint eigenvalue problems.
We reduce the a posteriori error estimates for approximate eigenvalues and approximate eigenspaces for eigenvalue problems to the a posteriori error estimates of the corresponding source problems. 
Thus the optimality results (quasi-orthogonality, contraction, and discrete reliability) for the corresponding source problems can be applied to analyzing the optimal convergence of AFEM for the eigenvalue problems.

The study in this paper is closely related to the work of \cite{feischl2014}.
By using the Gårding inequality $|a(u, u)| + C_{\operatorname{gard}} \|u\|_{0}^2 \geq \rho_{\operatorname{gard}} \|\nabla u\|_{0}^2\quad \forall u\in V$ in place of the coercivity condition (\ref{coercive})
and applying the results in Section~6.4 of \cite{feischl2014},
the conclusions of this paper remain valid for such a generalization.

\section*{Data availability}
Data will be made available on request.

\section*{Declaration of competing interest}
The authors declare that they have no known competing financial interests or personal relationships that could have appeared to influence the work reported in this paper.


\end{document}